\documentclass[11pt, notitlepage]{article}
\usepackage{amssymb,amsmath,comment}
\usepackage{amsthm}
\catcode`\@=11 \@addtoreset{equation}{section}
\def\thesection{\arabic{section}}

\def\theequation{\thesection.\arabic{equation}}
\catcode`\@=12
\usepackage{colortbl}
\usepackage{hyperref}
\usepackage[mathscr]{eucal}
\usepackage{epsf}
\usepackage{esint}
\usepackage{a4wide}
\def\R{\mathbb{R}}

\DeclareMathOperator*{\esssup}{ess\,sup}
\DeclareMathOperator*{\essinf}{ess\,inf}
\DeclareMathOperator*{\essliminf}{ess\,lim\,inf}
\DeclareMathOperator*{\esslimsup}{ess\,lim\,sup}

\DeclareMathOperator*{\supp}{supp}

\newcommand{\Om} {\Omega}

\newcommand{\noi} {\noindent}

\newcommand{\Tail} {\mathrm{Tail}}
\newcommand{\loc} {\mathrm{loc}}

\usepackage[all]{xy}
\catcode`\@=11
\def\theequation{\@arabic{\c@section}.\@arabic{\c@equation}}
\catcode`\@=12

\def\N{{I\!\!N}}

\newtheorem{Theorem}{Theorem}[section]
\newtheorem{Lemma}[Theorem]{Lemma}
\newtheorem{prop}[Theorem]{Proposition}
\newtheorem{Corollary}[Theorem]{Corollary}
\newtheorem{Remark}[Theorem]{Remark}
\newtheorem{Definition}[Theorem]{Definition}

\begin{document}

{\vspace{0.01in}}

\title{Some qualitative and quantitative properties of weak solutions to mixed anisotropic and nonlocal quasilinear elliptic and doubly nonlinear parabolic equations}

\author{Prashanta Garain}

\maketitle

\begin{abstract}\noindent
This article is divided into two parts. In the first part, we examine the Brezis-Oswald problem involving a mixed anisotropic and nonlocal $p$-Laplace operator. We establish results on existence, uniqueness, boundedness, and the strong maximum principle. Additionally, for certain mixed anisotropic and nonlocal $p$-Laplace equations, we prove a Sturmian comparison theorem, establish comparison and nonexistence results, derive a weighted Hardy-type inequality, and analyze a system of singular mixed anisotropic and nonlocal $p$-Laplace equations. A key component of our approach is the use of the Picone identity, which we adapt from the local and nonlocal cases. In the second part of the article, we focus on regularity estimates. In the elliptic setting, we establish a weak Harnack inequality and semicontinuity results. We also consider a class of doubly nonlinear mixed anisotropic and nonlocal parabolic equations, proving semicontinuity results and analyzing the pointwise behavior of solutions. These results rely on appropriate energy estimates, De Giorgi-type lemmas, and positivity expansions. Finally, we derive various energy estimates, which may be of independent interest.
\end{abstract}

\maketitle

\noi {Keywords: Elliptic and doubly nonlinear parabolic mixed anisotropic and nonlocal quasilinear equation, system of mixed singular equations, existence, regularity, weak Harnack inequality, semicontinuity, pointwise behavior, energy estimates, De Giorgi theory.}

\noi{\textit{2020 Mathematics Subject Classification: 35A01, 35B45, 35B50, 35B51, 35B65, 35J62, 35K59, 35M10, 35P30, 35R11.}

\bigskip

\tableofcontents

\section{Introduction}
This article is twofold. In the first part of the article, we establish several qualitative properties of solutions of the following mixed anisotropic and nonlocal quasilinear elliptic equations
\begin{equation}\label{meqn}
    -H_p u+(-\Delta_p)^s u=f(x,u)\text{ in }\Omega,
\end{equation}
for various type of nonlinearities $f$. We also study singular system of mixed equations of the form
\begin{equation}\label{smeqn}
\begin{split}
    -H_p u+(-\Delta_p)^s u&=F(x,u,v)\text{ in }\Omega\\
    -H_ p v+(-\Delta_p)^s v&=G(x,u,v)\text{ in }\Omega,
    \end{split}
\end{equation}
under some hypothesis on $F$ and $G$ having singularity to be made precise later on. We show that one of the key ingredient to obtain these results is the validity of the Picone identity, which has been already seen to be useful in the separate local and nonlocal setting.\\
Further, we prove comparison principles which are valid for the more general mixed elliptic operator
\begin{equation}\label{mo}
-H_p +(-\Delta_p)^s -H_q +(-\Delta_q)^s
\end{equation}
and mixed parabolic operator
\begin{equation}\label{mp}
\partial_t-H_p+(-\Delta_p)^s-H_q+(-\Delta_q)^s.
\end{equation}
In the second part of the article, we establish energy estimates for subsolutions of the equation \eqref{meqn} and also for the following doubly nonlinear mixed equation
\begin{equation}\label{mp1}
\partial_t(|u|^{\alpha-1}u)-H_p u+(-\Delta_p)^s u=g(x)|u|^{p-2}u\text{ in }\Omega
\end{equation}
under some hypothesis on $p,\alpha$ and $g$ to be made precise later on. We will observe that, Picone identity is useful to obtain these subsolution estimates. Further, we establish energy estimates for supersolutions of \eqref{meqn} and \eqref{mp1}.\\
Moreover, for the elliptic problem \eqref{meqn}, we establish weak Harnack inequality for supersolutions and semicontinuity results. Finally, for the parabolic problem \eqref{mp1}, we obtain semicontinuity and pointwise behavior of solutions.\\
Here $(-\Delta_p)^s$ is the fractional $p$-Laplace operator defined by
$$
(-\Delta_p)^s u(x)=\text{P.V.}\int_{\mathbb{R}^N}|u(x)-u(y)|^{p-2}(u(x)-u(y))\,d\mu,
$$
where $d\mu=|x-y|^{-N-ps}\,dx dy$ and P.V. denote the principal value. For more details on fractional operator, see \cite{Hitchhikersguide} and the references therein. Moreover, $H_p$ is the anisotropic $p$-Laplace operator defined by 
$$
H_p u=\text{div}(H(\nabla u)^{p-1}\nabla H(\nabla u)),
$$
where $H$ is a Finsler-Minkowski norm, that is $H:\mathbb{R}^N\to[0,\infty)$ is strictly convex such that $H\in C^1(\mathbb{R}^N\setminus\{0\})$,  and satisfy the following hypothesis:
\begin{enumerate}
\item[(H1)] $H(x)=0$ iff $x=0$.

    \item[(H2)] $H(tx)=|t|H(x)$ for every $x\in\mathbb{R}^N$, $t\in\mathbb{R}$,
    
    \item[(H3)] there exist constants $C_1,C_2>0$ such that $C_1|x|\leq H(x)\leq C_2|x|$ for every $x\in\mathbb{R}^N$.
\end{enumerate}
We present some examples now to give a little more insight on the anisotropic p-Laplace operator, refer to \cite{BFKzamp, MV, Xiathesis} and the references therein.\\
\noi \textbf{Examples:} Let $x=(x_1,x_2,\ldots,x_N)\in\mathbb{R}^N$,
\begin{enumerate}
\item[(i)] then, for $q>1$, we define 
\begin{equation}\label{ex11}
H_q(x):=\Big(\sum_{i=1}^{N}|x_i|^q\Big)^\frac{1}{q};
\end{equation}
\item[(ii)] for $\lambda,\mu>0$, we define
\begin{equation}\label{ex2}
H_{\lambda,\mu}(x):=\sqrt{\lambda\sqrt{\sum_{i=1}^{N}x_i^{4}}+\mu\sum_{i=1}^{N}x_i^{2}}.
\end{equation}
\end{enumerate}
Then, one may check from \cite{MV} that the functions $H_q, H_{\lambda,\mu}:\mathbb{R}^N\to[0,\infty)$ given by \eqref{ex11} and \eqref{ex2} are Finsler-Minkowski norm.

\begin{Remark}\label{exrmk1}
For $i=1,2$ if $\lambda_i,\mu_i$ are positive real numbers such that $\frac{\lambda_1}{\mu_1}\neq\frac{\lambda_2}{\mu_2}$, then $H_{\lambda_1,\mu_1}$ and $H_{\lambda_2,\mu_2}$ given by \eqref{ex2} defines two non-isometric norms in $\mathbb{R}^N$, see \cite{MV}.
\end{Remark}

\begin{Remark}\label{exrmk2}
Moreover, for $H=H_q$ given by \eqref{ex11} we have
\begin{equation}\label{ex}
H_{p} u=\sum_{i=1}^{N}\frac{\partial}{\partial x_i}\bigg(\Big(\sum_{k=1}^{N}\Big|\frac{\partial u}{\partial x_k}\Big|^{q}\Big)^\frac{p-q}{q}\Big|\frac{\partial u}{\partial x_i}\Big|^{q-2}\frac{\partial u}{\partial x_i}\bigg).
\end{equation}
Therefore, $H_p$ become the $p$-Laplace operator $\Delta_p$, when $q=2$ and pseudo $p$-Laplace operator $S_p$, when $q=p$ as given below: 
\begin{equation}\label{plap}
H_{p} u=
\begin{cases}
\Delta_p u:=\text{div}(|\nabla u|^{p-2}\nabla u),\,\text{if }q=2,\,1<p<\infty,\\
S_p u:=\sum_{i=1}^{N}\frac{\partial}{\partial x_i}\Big(|u_i|^{p-2}u_i\Big),\,\text{if }q=p\in(1,\infty),
\end{cases}
\end{equation}
where $u_i:=\frac{\partial u}{\partial x_i}$, for $i=1,2,\ldots,N$.
\end{Remark}
In particular, the above examples suggest that, our results hold for the following general class of mixed operators
$$
\sum_{i=1}^{N}\frac{\partial}{\partial x_i}\bigg(\Big(\sum_{k=1}^{N}\Big|\frac{\partial u}{\partial x_k}\Big|^{q}\Big)^\frac{p-q}{q}\Big|\frac{\partial u}{\partial x_i}\Big|^{q-2}\frac{\partial u}{\partial x_i}\bigg)+(-\Delta_p)^s u,
$$
which reduces to the well-known mixed operator $-\Delta u+(-\Delta)^s$. Equations related to this mixed operator has first been studied in \cite{Foondun} followed by \cite{Chen} using probability techniques. Recently mixed problems has been a topic of considerable attention and there is a colossal amount of literature available in this direction using purely analytic approach as well, see for example \cite{Biagicpde, Biagiped, Vecchihong, Valaa} and the references therein. We also refer to \cite{Valap} regarding physical interpretation of mixed local and nonlocal problems. To the best of our knowledge, most of our results are valid even for $p=2$, that is for the operator $-\Delta u+(-\Delta)^s$. Regarding the anisotropic mixed equation, as far we are aware, some regularity results are recently developed in \cite{GKansp, GKK, Kenta1, Kenta2} and also related eigenvalue and extremal problems are studied, see \cite{Gjms, GU}. We also refer to \cite{Cozzi, GK} and the references therein.\\
\textbf{Notations and assumptions:} Throughout the rest of the paper, we shall assume the following conditions and use the following notations unless otherwise mentioned.
\begin{itemize}
    \item $\Omega\subset\mathbb{R}^N$ will denote a bounded $C^1$ domain with $N\geq 2$.

    \item We assume that $1<p<\infty$,\, $0<s<1$ and $\alpha>0$.

    \item We denote by $d\mu=|x-y|^{-N-ps}\,dx dy$.

    \item For $f\in\mathbb{R}$, we define $f_\pm=\max\{\pm f,0\}$.

    \item We write $\sup\,u$ and $\inf\,u$ to denote the essential supremum and essential infimum of $u$ respectively.

    \item $C$ will denote a generic constant whose value may change from line to line or even in the same line. If $C$ depends on the parameters $l_1,l_2,\ldots,l_r$, then we write $C=C(l_1,l_2,\ldots,l_r)$ to mean the dependency of $C$ with these parameters.

    \item For a Lebesgue measurable set $S$, we denote $|S|$ to mean the Lebesgue measure of $S$.

    \item $B_r(x_0)$ denote a ball of radius centered at $x_0$.

    \item For three functions $f,g,h$, by the inequality $f\leq g\leq h$ on some Lebesgue measurable set $S$, we mean $f(x)\leq g(x)\leq h(x)$ for almost every $x\in S$.

    \item Throughout the paper $C_1,C_2$ will denote the constants considered in (H3).
\end{itemize}
\textbf{Functional Setting:} For the rest of the paper, let $\Omega\subset\mathbb{R}^N$ with $N\geq 2$ be a bounded $C^1$ domain unless otherwise mentioned. The Sobolev space $W^{1,p}(\Omega)$, $1<p<\infty$, is defined as the Banach space of locally integrable weakly differentiable functions
$u:\Omega\to\mathbb{R}$ equipped with the norm
\[
\|u\|_{W^{1,p}(\Omega)}=\| u\|_{L^p(\Omega)}+\|\nabla u\|_{L^p(\Omega)}.
\]
The space $W^{1,p}(\mathbb{R}^N)$ is defined analogously. The fractional Sobolev space $W^{s,p}(\Omega)$, $0<s<1<p<\infty$, is defined by
$$
W^{s,p}(\Omega)=\Big\{u\in L^p(\Omega):\frac{|u(x)-u(y)|}{|x-y|^{\frac{N}{p}+s}}\in L^p(\Omega\times \Omega)\Big\}
$$
and endowed with the norm
$$
\|u\|_{W^{s,p}(\Omega)}=\left(\int_{\Omega}|u(x)|^p\,dx+\int_{\Omega}\int_{\Omega}\frac{|u(x)-u(y)|^p}{|x-y|^{N+ps}}\,dx\,dy\right)^\frac{1}{p}.
$$
The space $W^{s,p}_{\mathrm{loc}}(\Omega)$ is defined analogously. We refer to \cite{Hitchhikersguide} and the references therein for more details on fractional Sobolev spaces. Due to the mixed behavior of our equations, following  \cite{VecchiBO, Vecchihong, Vecchihenon} and taking into account the definition of $H$, we consider the solution space
$$
W_0^{1,p}(\Omega)=\{u\in W^{1,p}(\mathbb{R}^N):u=0\text{ in }\mathbb{R}^N\setminus\Omega\}
$$
under the norm
$$
\|u\|_{W_0^{1,p}(\Omega)}=\left(\int_{\Omega}H(\nabla u)^p\,dx+\iint_{\mathbb{R}^{2N}}|u(x)-u(y)|^p\,d\mu\right)^\frac{1}{p}.
$$
The next result asserts that the standard Sobolev space is continuously embedded in
the fractional Sobolev space, see \cite[Proposition 2.2]{Hitchhikersguide}.
\begin{Lemma}\label{l1}
Let $\Omega$ be a smooth bounded domain in $\mathbb{R}^N$, $1<p<\infty,\,0<s<1$. Then there
exists a positive constant $C=C(N,p,s)$ such that $$
\|u\|_{W^{s,p}(\Omega)}\leq C\|u\|_{W^{1,p}(\Omega)}
$$
for every $u\in W^{1,p}(\Omega)$.
\end{Lemma}
Next, we have the following result from \cite[Lemma $2.1$]{Silva}.
\begin{Lemma}\label{locnon1}
There exists a constant $c=c(N,p,s,\Omega)$ such that
\begin{equation}\label{locnonsem}
\int_{\mathbb{R}^N}\int_{\mathbb{R}^N}\frac{|u(x)-u(y)|^p}{|x-y|^{N+ps}}\,dx\,dy\leq c\int_{\Omega}|\nabla u|^p\,dx
\end{equation}
for every $u\in W_0^{1,p}(\Omega)$. Here $u$ is extended by zero outside $\Omega$.
\end{Lemma}

In most of the results of Part-I, we deal with solutions in $W_0^{1,p}(\Omega)$ which is defined below.
\begin{Definition}\label{p1sol}
We say that $u\in W_{0}^{1,p}(\Omega)$ is a weak subsolution (or supersolution) of the problem \eqref{meqn} if for every nonnegative $\phi\in W_0^{1,p}(\Omega)$, we have
\begin{equation}\label{p1wksol}
\begin{split}
&\int_{\Omega}H(\nabla u)^{p-1}\nabla H(\nabla u)\nabla\phi\,dx+\iint_{\mathbb{R}^{2N}}|u(x)-u(y)|^{p-2}(u(x)-u(y))(\phi(x)-\phi(y))\,d\mu\\
&\quad-\int_{\Omega}f(x,u)\phi\,dx\leq (\text{ or }\,\geq)\,0,
\end{split}
\end{equation}
provided that the integral $\int_{\Omega}f(x,u)\phi\,dx$ above is finite. Further, we say that $u\in W_0^{1,p}(\Omega)$ is a weak solution of \eqref{meqn} if the equality in \eqref{p1wksol} holds for every $\phi\in W_0^{1,p}(\Omega)$ without any sign restriction. 

Analogously, the notion of weak subsolutions, supersolutions and solutions for the system \eqref{smeqn} and equations related to the non-homogeneous operator \eqref{mo} is defined.
\end{Definition}

\begin{Remark}\label{p1rmk}
We observe that if $u\in W_0^{1,p}(\Omega)$, the local and nonlocal integrals are finite. Indeed, due to Lemma \ref{Happ}-(C), we have
$$
\Big|\int_{\Omega}H(\nabla u)^{p-1}\nabla H(\nabla u)\nabla\phi\,dx\Big|\leq  C\int_{\Omega}H(\nabla u)^{p-1}|\nabla\phi|\,dx<\infty,
$$
using H\"older's inequality and that $u,\phi\in W_0^{1,p}(\Omega)$. Further, for the nonlocal integral, we observe that
\begin{equation*}
\begin{split}
&\iint_{\mathbb{R}^{2N}}|u(x)-u(y)|^{p-2}(u(x)-u(y))(\phi(x)-\phi(y))\,d\mu\\
&\leq \iint_{\mathbb{R}^{2N}}|u(x)-u(y)|^{p-1}|\phi(x)-\phi(y)|\,d\mu<\infty
\end{split}
\end{equation*}
using H\"older's inequality, that $u,\phi\in W_0^{1,p}(\Omega)$ and Lemma \ref{locnon1}.
\end{Remark}
In Part-II, we deal with local weak solutions and therefore, we define them later in Part-II in separate subsections.\\
\textbf{Auxiliary results:} Now, we state some useful results. The following results are taken \cite[Proposition 2.1]{PF20}, \cite[Proposition 1.2]{Xiathesis} and \cite[Lemma 5.9]{Heinonen}.
\begin{Lemma}\label{Happ}
For every $x\in\mathbb{R}^N\setminus\{0\}$ and $t\in\mathbb{R}\setminus\{0\}$ we have
\begin{enumerate}
\item[(A)] $x\cdot\nabla H(x)=H(x)$. 
\item[(B)] $\nabla H(tx)=\text{sign}(t)\nabla H(x)$.
\item[(C)] $|\nabla H(x)|\leq C$, for some fixed positive constant $C$.
\item[(D)] for every $x,y\in\mathbb{R}^N$ such that $x\neq y$
\begin{equation}\label{Falg}
\big(H(x)^{p-1}\nabla H(x)-H(y)^{p-1}\nabla H(y)\big)\cdot(x-y)>0.
\end{equation}
\end{enumerate}
\end{Lemma}
For $f,k\in\mathbb{R}$ and $\alpha>0$, we define
\begin{equation}\label{zeta}
\zeta_\pm(f,k)=\pm\alpha\int_{k}^{f}|s|^{\alpha-1}(s-k)_{\pm}\,ds.
\end{equation}
From \cite[Lemma 2.2]{LiaoJFA}, it follows that
\begin{Lemma}\label{Liaoin}
There exists a constant $\gamma=\gamma(\alpha)>0$ such that for any $f,k\in\mathbb{R}$, it holds that
\begin{equation}\label{Liaoin1}
\frac{1}{\gamma}(|f|+|k|)^{\alpha-1}(f-k)_\pm^{2}\leq \zeta_\pm(f,k)\leq\gamma(|f|+|k|)^{\alpha-1}(f-k)_\pm^{2}.
\end{equation}
\end{Lemma}
The following result follows from \cite[Proposition $3.1$ and Proposition $3.2$]{Dibe}.
\begin{Lemma}\label{Sobo}
Let $p,m\in[1,\infty)$ and $q=p(1+\frac{m}{N})$.
Assume that $\Om$ is a bounded smooth domain in $\mathbb{R}^N$. 
\begin{enumerate}
\item[(a)] If $u\in L^p\big(0,T;W^{1,p}(\Om)\big)\cap L^\infty\big(0,T;L^m(\Om)\big)$, then $u\in L^q\big(0,T;L^q(\Om)\big)$.
\item[(b)] Moreover, if $u\in L^p\big(0,T;W^{1,p}_{0}(\Om)\big)\cap L^\infty\big(0,T;L^m(\Om)\big)$, then there exists a constant $C=C(p,m,N)>0$ such that
\begin{equation}\label{Soboine}
\int_{0}^{T}\int_{\Om}|u(x,t)|^q\,dx\,dt\leq C\bigg(\int_{0}^{T}\int_{\Om}|\nabla u(x,t)|^p\,dx\,dt\bigg)\bigg(\sup_{0<t<T}\int_{\Om}|u(x,t)|^m\,dx\bigg)^\frac{p}{N}.
\end{equation}
\end{enumerate}
\end{Lemma}
For the following inequality, see \cite[Lemma 2.9]{BGK}.
\begin{Lemma}\label{Inequality1}
Let $a,b>0$, $\tau_1,\tau_2\geq 0$. Then for any $p>1$, there exists a constant $C=C(p)>1$ large enough such that
\begin{equation}\label{Ineeqn}
\begin{split}
&|b-a|^{p-2}(b-a)(\tau_1^{p}a^{-\epsilon}-\tau_2^{p}b^{-\epsilon})
\geq\frac{\zeta(\epsilon)}{C(p)}\Big|\tau_2 b^\frac{p-\epsilon-1}{p}-\tau_1 a^\frac{p-\epsilon-1}{p}\Big|^p\\
&\qquad-\Big(\zeta(\epsilon)+1+\frac{1}{\epsilon^{p-1}}\Big)\big|\tau_2-\tau_1\big|^p\big(b^{p-\epsilon-1}+a^{p-\epsilon-1}\big),
\end{split}
\end{equation}
where $0<\epsilon<p-1$ and $\zeta(\epsilon)=\epsilon(\frac{p}{p-\epsilon-1})^p$. If $0<p-\epsilon-1<1$, we may choose $\zeta(\epsilon)=\frac{\epsilon p^p}{p-\epsilon-1}$ in \eqref{Ineeqn}.
\end{Lemma}
The following real analysis lemma can be found in \cite[Lemma 4.1]{Dibe}.
\begin{Lemma}\label{iteration}
Let $(Y_j)_{j=0}^{\infty}$ be a sequence of positive real numbers such that
$Y_0\leq c_{0}^{-\frac{1}{\beta}}b^{-\frac{1}{\beta^2}}$ and $Y_{j+1}\leq c_0 b^{j} Y_j^{1+\beta}$,
$j=0,1,2,\dots$, for some constants $c_0,b>1$ and $\beta>0$. Then $\lim_{j\to\infty}\,Y_j=0$.
\end{Lemma}

The following version of the Gagliardo-Nirenberg-Sobolev inequality will be useful for us,
see \cite[Corollary 1.57]{Maly}.
\begin{Lemma}\label{Sobine}
Let $1<p<\infty$, $\Omega$ be an open set in $\R^n$ with $\lvert\Omega\rvert<\infty$ and
\begin{equation}\label{kappa}
\kappa=
\begin{cases}
\frac{n}{n-p},&\text{if}\quad 1<p<n,\\
2,&\text{if}\quad p\geq n.
\end{cases}
\end{equation}
There exists a positive constant $C=C(n,p)$ such that
\begin{equation}\label{e.friedrich}
\biggl(\int_\Omega \lvert u(x)\rvert^{\kappa p}\,dx\biggr)^{\frac{1}{\kappa p}}
\le C \lvert\Omega\rvert^{\frac{1}{n}-\frac{1}{p}+\frac{1}{\kappa p}} \biggl(\int_\Omega \lvert \nabla u(x)\rvert^{p}\,dx\biggr)^{\frac{1}{p}}
\end{equation}
for every $u\in W_0^{1,p}(\Omega)$.
\end{Lemma}

The following anisotropic Picone inequality follows from \cite[Lemma 2.2]{Jaros}. We also refer to \cite{BGM, Br}. 
\begin{Theorem}\label{BrPI}(\textbf{Anisotropic Picone inequality})
Let $H$ be a Finsler-Minkowski norm. Then for any differentiable functions $u,v$ in $\Omega$ with $u>0,\,v\geq 0$ in $\Omega$, we have
\begin{equation}\label{BrPIeqn}
\begin{split}
0&\leq H(\nabla v)^p-H(\nabla u)^{p-1}\nabla H(\nabla u)\nabla \Big(\frac{v^p}{u^{p-1}}\Big)\\
&=H(\nabla v)^p+(p-1)\Big(\frac{v}{u}\Big)^pH(\nabla u)^p-p\Big(\frac{v}{u}\Big)^{p-1}H(\nabla u)^{p-1}\nabla H(\nabla u)\nabla v\text{ in }\Omega
\end{split}
\end{equation}
Moreover, the equality in \eqref{BrPIeqn} holds in $\Omega$ iff $u=kv$ in $\Omega$ for some constant $k>0$.
\end{Theorem}

The following discrete Picone inequality follows from \cite[Theorem 2.4]{Jac}. We also refer to \cite{DP, Br}.
\begin{Theorem}\label{JacPI}(\textbf{Discrete Picone inequality})
Let $1<p<\infty$ and $\Omega$ be a domain in $\mathbb{R}^N$. Let $u,v$ be two nonnegative Lebesgue measurable functions in $\Omega$, with $u>0$ in $\Omega$ and non-constant. Then
\begin{equation}\label{JacPIeqn}
|u(x)-u(y)|^{p-2}\Big(\frac{v^p}{u^{p-1}}(x)-\frac{v^p}{u^{p-1}}(y)\Big)\leq |v(x)-v(y)|^p.
\end{equation}
Moreover, the equality in \eqref{JacPIeqn} holds in $\Omega$ iff $v=ku$ for some constant $k>0$.
\end{Theorem}

\section{Part-I}
\subsection{Brezis-Oswald problem}
In the celebrated work \cite{BOna}, Brezis and Oswald studied the Laplace equation
\begin{equation}\label{BOeqncel}
\begin{split}
-\Delta u&=f(x,u)\text{ in }\Omega\\
u&\gneqq 0\text{ in }\Omega,\quad
u=0\text{ in }\mathbb{R}^N\setminus\Omega,
\end{split}
\end{equation}
with some source function $g$. This has been extended to the $p$-Laplace operator $-\Delta_p$ in \cite{DiazSaa, Dm}, for the nonlocal operator $(-\Delta_p)^s$ in \cite{BOnonloc}. Recently, it has been extended to the mixed local and nonlocal operator $-\Delta_p u+(-\Delta_p)^s u$ in \cite{VecchiBO2, VecchiBO} and the authors proved existence, uniqueness, boundedness result and strong maximum principle. We also refer to \cite{Zhangsmp} for strong maximum principle in the mixed case. Here, in the mixed anisotropic setting, we establish strong maximum principle, boundedness, uniqueness and existence result for the mixed anisotropic and nonlocal Brezis-Oswald problem \eqref{BOeqn} given by
\begin{equation}\label{BOeqn}
\begin{split}
-H_p u+(-\Delta_p)^s u&=f(x,u)\text{ in }\Omega\\
u&\gneqq 0\text{ in }\Omega,\quad
u=0\text{ in }\mathbb{R}^N\setminus\Omega,
\end{split}
\end{equation}
where $\Omega$ is a bounded $C^1$ domain in $\mathbb{R}^N$ with $N\geq 2$ and $f$ satisfy the following hypothesis:
\begin{enumerate}
\item[(f1)] $f:\Omega\times[0,\infty)\to\mathbb{R}$ is a {\color{blue}Carath\'eodory} function.

\item[(f2)] $f(\cdot,t)\in L^\infty(\Omega)$ for every $t\geq 0$.

\item[(f3)] there exists a constant $c_p>0$ such that
$$
|f(x,t)|\leq c_p(1+t^{p-1})\text{ for a.e. }x\in\Omega\text{ and every }t\geq 0.
$$

\item[(f4)] For a.e. $x\in\Omega$, the function $t\mapsto \frac{f(x,t)}{t^{p-1}}$ is strictly decreasing in $(0,\infty)$.

\item[(f5)] there exists $\rho_f>0$ such that $f(x,t)>0$ for a.e. $x\in\Omega$ and every $0<t<\rho_f$.
\end{enumerate}
An example of a class of function $f$ satisfying $(f1)-(f5)$ is $f(x,u)=u^q$ for $0\leq q\leq p-1$. Following Definition 2.1 of \cite{VecchiBO}, we define the notion of solutions of the problem \eqref{BOeqn} as follows:
\begin{Definition}\label{BOeqnwksol}
We say that $u\in W_0^{1,p}(\Omega)$ is a weak solution of \eqref{BOeqn} if 
\begin{enumerate}
\item[(i)]
\begin{equation}\label{BOwksoleqn}
\begin{split}
&\int_{\Omega}H(\nabla u)^{p-1}\nabla H(\nabla u)\nabla\phi\,dx+\iint_{\mathbb{R}^{2N}}|u(x)-u(y)|^{p-2}(u(x)-u(y))(\phi(x)-\phi(y))\,d\mu\\
&\quad\quad=\int_{\Omega}f(x,u)\phi\,dx,
\end{split}
\end{equation}
holds for every $\phi\in W_0^{1,p}(\Omega)$ and 
\item[(ii)] $u\geq 0$ in $\Omega$ and $|x\in\Omega:u(x)>0|>0$.
\end{enumerate}
\end{Definition}

First we obtain the following strong maximum principle for the problem \eqref{BOeqn} below.

\begin{Theorem}\label{smp}(\textbf{Strong maximum principle})
 Let $f$ satisfy $(f1)-(f3)$ above and $u\in W_0^{1,p}(\Omega)$ such that $u\geq 0$ in $\Omega$ satisfy \eqref{BOwksoleqn} for every $\phi\in W_0^{1,p}(\Omega)$. Then either $u\equiv 0$ or $u>0$ in $\Omega$.  
 \end{Theorem}  

\begin{proof}
We follow the proof of \cite[Theorem 3.1]{VecchiBO}. To this end, let $\epsilon>0$ and $\phi$ be the same function as in the proof of \cite[Theorem 3.1, page 8]{VecchiBO} and then choosing $\phi_\epsilon=\frac{\phi^p}{(u+\epsilon)^{p-1}}$ as a test function in \eqref{BOwksoleqn}, we obtain
\begin{equation}\label{BOeqnwksoltst1}
\begin{split}
&\int_{\Omega}H(\nabla u)^{p-1}\nabla H(\nabla u)\nabla\phi_\epsilon\,dx+\iint_{\mathbb{R}^{2N}}{|u(x)-u(y)|^{p-2}(u(x)-u(y))(\phi_\epsilon(x)-\phi_\epsilon(y)) }\,d\mu\\
&\quad=\int_{\Omega}f(x,u)\phi_\epsilon\,dx.
\end{split}
\end{equation}
We estimate the first integral in the L.H.S. above only, since the other integrals are estimated in the proof of \cite[Theorem 3.1]{VecchiBO}. To this end, by Lemma \ref{Happ}-(A), we observe that
\begin{equation*}
\begin{split}
H(\nabla u)^{p-1}\nabla H(\nabla u)\nabla\phi_\epsilon&=-(p-1)\Big(\frac{\phi}{u+\epsilon}\Big)^pH(\nabla u)^p+p\Big(\frac{\phi}{u+\epsilon}\Big)^{p-1}H(\nabla u)^{p-1}\nabla H(\nabla u)\nabla\phi.
\end{split}
\end{equation*}
Using the above estimate in \eqref{BOeqnwksoltst1} along with the property (H3) of $H$ and Lemma \ref{Happ}-(C), we obtain
\begin{equation}\label{BOeqnwksoltst2}
\begin{split}
C_1(p-1)\int_{\Omega}\Big(\frac{\phi}{u+\epsilon}\Big)^p|\nabla u|^p\,dx&\leq  \iint_{\mathbb{R}^{2N}}\frac{|u(x)-u(y)|^{p-2}(u(x)-u(y))(\phi_\epsilon(x)-\phi_\epsilon(y)) }{|x-y|^{N+ps}}\,dx dy\\
&+C_2\,p\Big(\frac{\phi}{u+\epsilon}\Big)^{p-1}|\nabla u|^{p-1}|\nabla\phi|\,dx-\int_{\Omega}f(x,u)\phi_\epsilon\,dx
\end{split}
\end{equation}
for some constants $C_1,C_2>0$ independent of $\epsilon$. The rest of the proof follows analogous to the proof of \cite[Theorem 3.1]{VecchiBO}.
\end{proof}

As a consequence of Theorem \ref{smp}, we have
\begin{Corollary}\label{smpcor1}
Let $f$ satisfy the hypothesis $(f1)-(f3)$ and $u\in W_0^{1,p}(\Omega)$ be a weak solution of the problem \eqref{BOeqn}. Then $u>0$ in $\Omega$.
\end{Corollary}

\begin{Remark}\label{BOeqnrmk1}
Further, as in \cite[Remark 3.4]{VecchiBO}, by carefully observing the proof of Theorem \ref{smp}, one has Theorem \ref{smp} valid for weak solutions of the boundary value problem
\begin{equation}\label{BOeqn1smp}
\begin{split}
-H_p u+(-\Delta_p)^s u=g(x,u)\text{ in }\Omega\\
u\gneqq 0\text{ in }\Omega,\quad
u=0\text{ in }\mathbb{R}^N\setminus\Omega,
\end{split}
\end{equation}
where $\Omega$ is a bounded $C^1$ domain in $\mathbb{R}^N$ with $N\geq 2$ and $g:\Omega\times\mathbb{R}\to\mathbb{R}$ is a Carath\'eodory function satisfying
\begin{enumerate}
    \item[(1)] $g(x,0)\geq 0$ for a.e. $x\in\Omega$

    \item[(2)] $g(x,t)\geq -c_f t^{p-1}$ for a.e. $x\in\Omega$ and every $0<t<1$

    \item[(3)] $g(x,t)\geq c_p(1+t^{p-1})$ for a.e. $x\in\Omega$ and every $t\geq 1$

    \item[(4)] there exists a constant $c_p>0$ such that
    \begin{equation}\label{gcn}
    |g(x,t)|\leq c_p(1+t^{q-1})\text{ for a.e. }x\in\Omega\text{ and every }t\geq 0, 
    \end{equation}
    where $1\leq q\leq\frac{Np}{N-p}$ if $1<p<N$ and $p\leq q<\infty$ if $p\geq N$.
\end{enumerate}
An example of a function $g$ satisfying (1)-(3) above is $f(x,t)=(-a(x)+\lambda)|t|^{p-2}t$, where $\lambda\in\mathbb{R}$ and $a\in L^\infty(\Omega)$. 
\end{Remark}

Taking into account the definition of $H$ and Lemma \ref{Happ} and proceeding along the lines of the proof of \cite[Theorem 4.1]{VecchiBO}, the following result holds.
\begin{Theorem}\label{boTHM2}(\textbf{Boundedness})
Let $f$ satisfy $(f1)-(f3)$ and suppose $u\in W_0^{1,p}(\Omega)$ is a weak solution of \eqref{BOeqn}. Then $u\in L^\infty(\Omega)$.
\end{Theorem}

\begin{Theorem}\label{boTHM3}(\textbf{Uniqueness})
Let $f$ satisfy $(f1)-(f5)$. Then the problem \eqref{BOeqn} admits at most one weak solution in $W_0^{1,p}(\Omega)$.
\end{Theorem}
\begin{proof}
We follow the proof of \cite[Theorem 4.3]{VecchiBO}. Let $u,v\in W_0^{1,p}(\Omega)$ be two weak solutions of the problem \eqref{BOeqn}. Then by Theorem \ref{boTHM2}, we have $u,v\in L^\infty(\Omega)$. We fix $\epsilon>0$ and set
$$
\phi_\epsilon=\frac{v^p}{(u+\epsilon)^{p-1}}-u,\quad \psi_\epsilon=\frac{u^p}{(v+\epsilon)^{p-1}}-v.
$$
Since $u,v\in W_0^{1,p}(\Omega)\cap L^\infty(\Omega)$ and $u,v\geq 0$, so $\phi_\epsilon,\psi_\epsilon\in W_0^{1,p}(\Omega)$ for every $\epsilon>0$. Therefore, using $\phi_\epsilon$ and $\psi_\epsilon$ as test functions in \eqref{BOwksoleqn} and adding the resulting equations, we obtain
\begin{equation}\label{BOeqn1}
\begin{split}
&\int_{\Omega}H(\nabla u)^{p-1}\nabla H(\nabla u)\nabla\phi_\epsilon+\int_{\Omega}H(\nabla u)^{p-1}\nabla H(\nabla u)\nabla\psi_\epsilon\\
&\quad+\iint_{\mathbb{R}^{2N}}{|u(x)-u(y)|^{p-2}(u(x)-u(y))(\phi_\epsilon(x)-\phi_\epsilon(y))}\,d\mu\\
&\qquad+\iint_{\mathbb{R}^{2N}}{|u(x)-u(y)|^{p-2}(u(x)-u(y))(\phi_\epsilon(x)-\phi_\epsilon(y))}\,d\mu\\
&\qquad=\int_{\Omega}(f(x,u)\phi_\epsilon+f(x,v)\psi_\epsilon)\,dx.
\end{split}
\end{equation}
Next, we compute the sign of the sum of the first two integrals above. To this end, we observe that
\begin{equation}\label{noname1}
\begin{aligned}
&H(\nabla u)^{p-1}\nabla H(\nabla u)\nabla\phi_\epsilon\\
&=-H(\nabla u)^p+p\Big(\frac{v}{u+\epsilon}\Big)^{p-1}H(\nabla u)^{p-1}\nabla H(\nabla u)\nabla v-(p-1)\Big(\frac{v}{u+\epsilon}\Big)^p H(\nabla u)^p
\end{aligned}
\end{equation}
and
\begin{equation}\label{noname2}
\begin{aligned}
&H(\nabla v)^{p-1}\nabla H(\nabla v)\nabla\psi_\epsilon\\
&=-H(\nabla v)^p+p\Big(\frac{u}{v+\epsilon}\Big)^{p-1}H(\nabla v)^{p-1}\nabla H(\nabla v)\nabla u-(p-1)\Big(\frac{u}{v+\epsilon}\Big)^p H(\nabla v)^p.
\end{aligned}
\end{equation}
Adding \eqref{noname1} and \eqref{noname2}, we obtain
\begin{equation}\label{noname3}
\begin{aligned}
&\int_{\Omega}(H(\nabla u)^{p-1}\nabla H(\nabla u)\nabla\phi_\epsilon+H(\nabla v)^{p-1}\nabla H(\nabla v)\nabla\psi_\epsilon)\,dx\\
&=\int_{\Omega}(-H(\nabla u)^p+p\Big(\frac{v}{u+\epsilon}\Big)^{p-1}H(\nabla u)^{p-1}\nabla H(\nabla u)\nabla v-(p-1)\Big(\frac{v}{u+\epsilon}\Big)^p H(\nabla u)^p)\,dx\\
&+\int_{\Omega}(-H(\nabla v)^p+p\Big(\frac{u}{v+\epsilon}\Big)^{p-1}H(\nabla v)^{p-1}\nabla H(\nabla v)\nabla u-(p-1)\Big(\frac{u}{v+\epsilon}\Big)^p H(\nabla v)^p)\,dx.\\
&=\int_{\Omega}( -H(\nabla u)^p+p\Big(\frac{u}{v+\epsilon}\Big)^{p-1}H(\nabla v)^{p-1}\nabla H(\nabla v)\nabla u-(p-1)\Big(\frac{u}{v+\epsilon})^pH(\nabla v)^p)\,dx\\
&+\int_{\Omega}(-H(\nabla v)^p+ p\Big(\frac{v}{u+\epsilon}\Big)^{p-1}H(\nabla u)^{p-1}\nabla H(\nabla u)\nabla v-(p-1)\Big(\frac{v}{u+\epsilon})^p H(\nabla u)^p)\,dx\\
&\leq 0,
\end{aligned}
\end{equation}
which follows from Theorem \ref{BrPI}. Now using the above nonpositivity in \eqref{BOeqn1}, the result follows proceeding along the lines of the proof of \cite[Theorem 4.3]{VecchiBO}.
\end{proof}
Next, we study the existence result for the problem \eqref{BOeqn}. To this end, first we study the following eigenvalue problem as in \cite{VecchiBO}.\\
\textbf{Eigenvalue problem:} Consider the eigenvalue problem
\begin{equation}\label{boevp}
\begin{split}
-H_p u+(-\Delta_p)^s u+a(x)|u|^{p-2}u&=\beta|u|^{p-2}u\text{ in }\Omega\\
u\neq 0\text{ in }\Omega,\quad u=0\text{ in }\mathbb{R}^N\setminus{\Omega},
\end{split}
\end{equation}
where $a\in L^\infty(\Omega)$ and $\beta\in\mathbb{R}$.
\begin{prop}\label{boevpprop}
Let $a\in L^\infty(\Omega)$, then the problem \eqref{boevp} admits a smallest eigenvalue $\beta_1\in\mathbb{R}$ which is simple and whose associated eigenfunctions do not change sign in $\mathbb{R}^N$. Moreover, every eigenfunction associated to the eigenvalue $\beta$ such that $\beta>\beta_1$ is nodal, that is changes sign in $\Omega$.
\end{prop}
\begin{proof}
We define
$$
\beta_1:=\beta_1(-H_p+(-\Delta_p)^s+a)=\inf\{\gamma(u):u\in M\},
$$
where $\gamma:W_0^{1,p}(\Omega)\to\mathbb{R}$ is the $C^1$-functional defined as
$$
\gamma(u)=\int_{\Omega}H(\nabla u)^p\,dx+\iint_{\mathbb{R}^{2N}}|u(x)-u(y)|^p\,d\mu+\int_{\Omega}a(x)|u|^p\,dx
$$
for all $u\in W_0^{1,p}(\Omega)$ and let it be constrained on the $C^1$-Banach manifold
$$
M:=\Big\{u\in W_0^{1,p}(\Omega):\int_{\Omega}|u|^p\,dx=1\Big\}.
$$
Now taking into account the properties of $H$, Theorem \ref{smp}, Theorem \ref{boTHM2} along with Remark \ref{BOeqnrmk1}, Theorem \ref{BrPI}, Theorem \ref{JacPI} and proceeding along the lines of the proof of \cite[Proposition 5.1]{VecchiBO}, the result follows.
\end{proof}
Before stating the existence result, let us fix some notations and assumptions throughout this subsection: Let $f$ satisfy the hypothesis $(f1)-(f5)$. Using $(f4)$, we introduce the functions
\begin{equation}\label{a01}
a_0(x):=\lim_{t\to 0^+}\frac{f(x,t)}{t^{p-1}},\quad a_\infty(a):=\lim_{t\to+\infty}\frac{f(x,t)}{t^{p-1}}
\end{equation}
We observe that $a_0$ may be unbounded, but by $(f3)$, the function $a_\infty$ is bounded. Now we introduce
\begin{equation}\label{ev1}
\begin{split}
&\beta_1(-H_p+(-\Delta_p)^s-a_0)\\
&:=\inf_{u\in W_0^{1,p}(\Omega),\,\|u\|_{L^p(\Omega)}=1}\left\{\int_{\Omega}H(\nabla u)^p\,dx+\iint_{\mathbb{R}^{2N}}|u(x)-u(y)|^p\,d\mu-\int_{\{u\neq 0\}}a_0|u|^p\,dx\right\}
\end{split}
\end{equation}
and
\begin{equation}\label{ev2}
\begin{split}
&\beta_1(-H_p+(-\Delta_p)^s-a_\infty)\\
&:=\inf_{u\in W_0^{1,p}(\Omega),\,\|u\|_{L^p(\Omega)}=1}\left\{\int_{\Omega}H(\nabla u)^p\,dx+\iint_{\mathbb{R}^{2N}}|u(x)-u(y)|^p\,d\mu-\int_{\Omega}a_\infty|u|^p\,dx\right\}.
\end{split}
\end{equation}
Finally, we have the following existence result of this subsection.
\begin{Theorem}\label{evpex}(\textbf{Existence})
Let $f$ satisfy $(f1)-(f5)$. Then there exists a positive weak solution to \eqref{BOeqn} iff 
$$
\beta_1(-H_p+(-\Delta_p)^s-a_0)<0<\beta_1(-H_p+(-\Delta_p)^s-a_\infty).
$$
\end{Theorem}
\begin{proof}
The proof follows by the combination of Proposition \ref{evpexp1} and Proposition \ref{evpexp2} below.
\end{proof}
We set the functional $E:W_0^{1,p}(\Omega)\to\mathbb{R}$ defined by
\begin{equation}\label{E}
E(u):=\int_{\Omega}H(\nabla u)^p\,dx+\iint_{\mathbb{R}^{2N}}|u(x)-u(y)|^p\,d\mu-\int_{\Omega}F(x,u)\,dx,
\end{equation}
where
$$
F(x,u)=\int_{0}^{u}f(x,t)\,dt.
$$
Then we observe that the functional $E$ is well defined, is differentiable and its critical points are the weak solution of the problem \eqref{BOeqn}. Taking into account the properties of $H$ and following the lines of the proof of \cite[Proposition 6.2]{VecchiBO}, the following result follows.
\begin{prop}\label{evpexp1}
Let $f$ satisfy $(f1)-(f5)$ and $E$ be the functional defined in \eqref{E} above and assume that
$$
\beta_1(-H_p+(-\Delta_p)^s-a_0)<0<\beta_1(-H_p+(-\Delta_p)^s-a_\infty).
$$
Then the following hold:
\begin{enumerate}
    \item[(a)] $E$ is coercive on $W_0^{1,p}(\Omega)$.
    \item[(b)] $E$ is weakly lower semicontinuous in $W_0^{1,p}(\Omega)$, therefore, it has a minimum $v\in W_0^{1,p}(\Omega)$.
    \item[(c)] There exists $\psi\in W_0^{1,p}(\Omega)$ which satisfy $E(\psi)<0$, so that
    $$
    \min_{u\in W_0^{1,p}(\Omega)}E(u)<0,
    $$
    and
    $u=|v|$ is a weak solution of \eqref{BOeqn}.
\end{enumerate}
\end{prop}
The result below follows by arguing similarly as in the proof of \cite[Lemma 6.3]{VecchiBO} and \cite[Theorem 1.3]{VecchiBO2}, by taking into account the properties of $H$ along with Theorem \ref{smp}, Theorem \ref{boTHM2}, Proposition \ref{boevpprop} along with Theorem \ref{BrPI} and Theorem \ref{JacPI}.
\begin{prop}\label{evpexp2}
Let $f$ satisfy $(f1)-(f5)$. If $u\in W_0^{1,p}(\Omega)$ is a weak solution of \eqref{BOeqn}, then
$$
\beta_1(-H_p+(-\Delta_p)^s-a_0)<0<\beta_1(-H_p+(-\Delta_p)^s-a_\infty).
$$
\end{prop}

\subsection{Comparison results}
First, we prove the following Sturmian comparison theorem, which has been studied in the local case in \cite{AlP} and the nonlocal case is settled in \cite{Jac}. Mixed case is unknown to our knowledge till date. The proof is based on the use of Picone's inequality and choosing of a suitable test function.
\begin{Theorem}\label{scp}(\textbf{Sturmian comparison principle})
Let $a_1$ and $a_2$ be two continuous functions such that $a_1<a_2$ in $\Omega$. Suppose $u\in W_0^{1,p}(\Omega)\cap L^\infty(\Omega)$ is a weak solution of 
\begin{equation}\label{scpeqn1}
-H_p u+(-\Delta_p)^s u=a_1(x)u^{p-1}\text{ in }\Omega,\quad u>0\text{ in }\Omega
\end{equation}
as in Definition \ref{p1sol}. Then any nontrivial weak solution $v\in W_0^{1,p}(\Omega)$ of the equation
\begin{equation}\label{scpeqn2}
-H_p v+(-\Delta_p)^s v=a_2(x)|v|^{p-2}v\text{ in }\Omega
\end{equation}
as in Definition \ref{p1sol} must vanish in $\Omega$.
\end{Theorem}

\begin{proof}
Suppose $v$ does not vanish in $\Omega$. Without loss of generality, let $v>0$ in $\Omega$, since otherwise the case $v<0$ in $\Omega$ can be dealt similarly, by working with $w=-v$ in place of $v$. Let $\epsilon>0$ and define $v_\epsilon=v+\epsilon$, $\phi_\epsilon=\frac{u^p}{(v+\epsilon)^{p-1}}$. Then we observe that $\phi_\epsilon\in W_0^{1,p}(\Omega)$ for every $\epsilon>0$. By Theorem \ref{BrPI} and Theorem \ref{JacPI}, we obtain
\begin{equation}\label{scpeqn3}
\begin{split}
0&\leq \int_{\Omega}H(\nabla u)^p\,dx-\int_{\Omega}H(\nabla v_\epsilon)^{p-1}\nabla H(\nabla v_\epsilon)\nabla\phi_{\epsilon}\,dx\\
&+\iint_{\mathbb{R}^{2N}}{|u(x)-u(y)|^p}\,d\mu-\iint_{\mathbb{R}^{2N}}{|v_\epsilon(x)-v_\epsilon(y)|^{p-2}(v_\epsilon(x)-v_\epsilon(y))(\phi_\epsilon(x)-\phi_\epsilon(y))}\,d\mu\\
&=\Big(\int_{\Omega}H(\nabla u)^p\,dx+\iint_{\mathbb{R}^{2N}}{|u(x)-u(y)|^p}\,d\mu\Big)\\
&-\Big(\int_{\Omega}H(\nabla v_\epsilon)^{p-1}\nabla H(\nabla v_\epsilon)\nabla\phi_{\epsilon}\,dx+\iint_{\mathbb{R}^{2N}}{|v_\epsilon(x)-v_\epsilon(y)|^{p-2}(v_\epsilon(x)-v_\epsilon(y))(\phi_\epsilon(x)-\phi_\epsilon(y))}\,d\mu\Big)\\
&= \int_{\Omega}a_1(x)\,u^p\,dx-\int_{\Omega}a_2(x)\Big(\frac{v}{v+\epsilon}\Big)^{p-1}u^p\,dx.
\end{split}
\end{equation}
The last equality above is obtained by choosing $u$ as a test function in \eqref{scpeqn1} and $\phi_\epsilon$ in \eqref{scpeqn2} respectively. Passing to the limit as $\epsilon\to 0^+$ and using Fatou's lemma, we obtain the inequality
$
0\leq \int_{\Omega}(a_1-a_2)u^p\,dx.
$
But $a_1<a_2$ in $\Omega$ gives $\int_{\Omega}(a_1-a_2)u^p\,dx<0$, which is a contradiction to the above inequality. Hence, $v$ mush vanish in $\Omega$.
\end{proof}

The following two comparison results are proved in the mixed case $-\Delta_p-\Delta_{J,p}$ in \cite[Lemma 2.1 and Corollary 2.1]{Rossi}, where
\begin{equation}\label{djp}
\Delta_{J,p}u:=2\int_{\mathbb{R}^N}|u(x)-u(y)|^{p-2}(u(x)-u(y))J(x-y)\,dy,
\end{equation}
where the kernel $J:\mathbb{R}^N\to\mathbb{R}$ is a radially symmetric, nonnegative continuous function with compact support, $J(0)>0$ and $\int_{\mathbb{R}^N}J(\tau)\,d\tau=1$. Here, we prove such comparison results for both elliptic and parabolic mixed equations covering a more general class of nonhomogeneous operator including $-H_p+(-\Delta_p)^s$ as will be shown below. 
\begin{Theorem}\label{comp1}(\textbf{Comparison principle})
Let $u,v\in W^{1,p}(\mathbb{R}^N)$ be such that 
\begin{equation}\label{comp1eqn1}
-H_p u+(-\Delta_p)^s u\leq -H_p v+(-\Delta_p)^s v\text{ in }\Omega
\end{equation}
in the weak sense, that is
\begin{equation}\label{cpwksol1}
\begin{split}
&\int_{\Omega}H(\nabla u)^{p-1}\nabla H(\nabla u)\nabla\phi\,dx+\iint_{\mathbb{R}^{2N}}|u(x)-u(y)|^{p-2}(u(x)-u(y))(\phi(x)-\phi(y))\,d\mu\\
&\quad\leq \int_{\Omega}H(\nabla v)^{p-1}\nabla H(\nabla v)\nabla\phi\,dx+\iint_{\mathbb{R}^{2N}}|v(x)-v(y)|^{p-2}(v(x)-v(y))(\phi(x)-\phi(y))\,d\mu 
\end{split}
\end{equation}
for every $\phi\in W_0^{1,p}(\Omega)$, $\phi\geq 0$, then $u\leq v$ in $\Omega$ if $u\leq v$ in $\mathbb{R}^N\setminus\Omega$.
\end{Theorem}


\begin{proof}
Let $u\leq v$ in $\mathbb{R}^N\setminus\Omega$, then we can choose $\phi=\max\{u-v,0\}$ as a test function in \eqref{cpwksol1} and obtain
\begin{equation}\label{comp1eqn2}
\begin{split}
0&\leq \int_{\Omega}\big(H(\nabla v)^{p-1}\nabla H(\nabla v)-H(\nabla u)^{p-1}\nabla H(\nabla u)\big)\nabla\phi\,dx\\
&\qquad+\iint_{\mathbb{R}^{2N}}\big(|v(x)-v(y)|^{p-2}(v(x)-v(y))-|u(x)-u(y)|^{p-2}(u(x)-u(y))\big)(\phi(x)-\phi(y))\,d\mu.
\end{split}
\end{equation}
Arguing as in \cite[Lemma 9]{LL}, we get the nonlocal integral above is $\leq 0$ that is
$$
\iint_{\mathbb{R}^{2N}}\Big(|v(x)-v(y)|^{p-2}(v(x)-v(y))-|u(x)-u(y)|^{p-2}(u(x)-u(y))\Big)(\phi(x)-\phi(y))\,d\mu\leq 0.
$$
Next, by Theorem \ref{Falg}, the local integral
$$
\int_{\Omega}\big(H(\nabla v)^{p-1}\nabla H(\nabla v)-H(\nabla u)^{p-1}\nabla H(\nabla u)\big)\nabla \phi\,dx\leq 0.
$$
Hence we get that
$$
\int_{\Omega}\big(H(\nabla v)^{p-1}\nabla H(\nabla v)-H(\nabla u)^{p-1}\nabla H(\nabla u)\big)\nabla \phi\,dx=0,
$$
which gives $\phi=0$ in $\Omega$ that is $u\leq v$  in $\Omega$.
\end{proof}
\begin{Corollary}\label{app1}
By Theorem \ref{comp1}, we obtain that if $u\in W^{1,p}(\mathbb{R}^N)$ such that 
$$
-H_p u+(-\Delta_p)^s u\geq 0\text{ in }\Omega
$$
in the weak sense, then $u\geq 0$ in $\Omega$ if $u\geq 0$ in $\mathbb{R}^N\setminus\Omega$.
\end{Corollary}
Following the same proof as of Theorem \ref{comp1}, we obtain the following comparison results including for the nonhomogeneous $(p,q)$-elliptic and parabolic mixed problems. 
\begin{Theorem}\label{rmk1}
Let $1<p,q<\infty$, $0<s<1$ and $u,v\in W^{1,p}(\mathbb{R}^N)\cap W^{1,q}(\mathbb{R}^N)$ be such that 
\begin{equation}\label{comp1eqn32}
-H_p u+(-\Delta_p)^s u-H_q u+(-\Delta_q)^s u\leq -H_p v+(-\Delta_p)^s v-H_q v+(-\Delta_q)^s v\text{ in }\Omega
\end{equation}
in the weak sense, then $u\leq v$ in $\Omega$ if $u\leq v$ in $\mathbb{R}^N\setminus\Omega$.
\end{Theorem}

\begin{Theorem}\label{rmkn123}
Let $u,v\in C(0,T; L^2(\Omega))\cap W^{1,p}(\mathbb{R}^N\times(0,T))$ be such that 
\begin{equation}\label{comp1eqn34}
\partial_t u-H_p u+(-\Delta_p)^s u\leq \partial_t v-H_p v+(-\Delta_p)^s v\text{ in }\Omega
\end{equation}
in the weak sense, then $u\leq v$ in $\Omega\times(0,T)$ if $u\leq v$ in $\big(\mathbb{R}^N\setminus\Omega\big)\times(0,T)$ and $u(x,0)\leq v(x,0)$ in $\Omega$.
\end{Theorem}

\begin{Theorem}\label{rmkn12}
Let $1<p,q<\infty$, $0<s<1$ and $u,v\in C(0,T; L^2(\Omega))\cap W^{1,p}(\mathbb{R}^N\times(0,T))\cap W^{1,q}(\mathbb{R}^N\times(0,T))$ be such that 
\begin{equation}\label{comp1eqn3}
\partial_t u-H_p u+(-\Delta_p)^s u- H_q u+(-\Delta_q)^s u\leq \partial_t v-H_p v+(-\Delta_p)^s v-H_q v+(-\Delta_q)^s v\text{ in }\Omega
\end{equation}
in the weak sense, then $u\leq v$ in $\Omega\times(0,T)$ if $u\leq v$ in $\big(\mathbb{R}^N\setminus\Omega\big)\times(0,T)$ and $u(x,0)\leq v(x,0)$ in $\Omega$.
\end{Theorem}

\begin{Remark}\label{tstrmk}
We remark that the definition of weak sense in Theorem \ref{rmkn123} and Theorem \ref{rmkn12} is analogous to the doubly nonlinear equation below in subsection \ref{dls} and the admissibility of the test function $\phi=\max\{u-v,0\}$ in Theorem \ref{rmkn123} and Theorem \ref{rmkn12} can be justified using the mollification technique, for example in \cite{Bog1, Kenta2}. 
\end{Remark}

\begin{Corollary}\label{app2}
As an application of Theorem \ref{rmk1}, we obtain that if $1<p,q<\infty$, $0<s<1$ and $u\in W^{1,p}(\mathbb{R}^N)\cap W^{1,q}(\mathbb{R}^N)$ such that 
$$
-H_p u+(-\Delta_p)^s u-H_q u+(-\Delta_q)^s u\geq 0\text{ in }\Omega,
$$
in the weak sense, then $u\geq 0$ in $\Omega$ if $u\geq 0$ in $\mathbb{R}^N\setminus\Omega$.
\end{Corollary}

\begin{Corollary}\label{app23}
As an application of Theorem \ref{rmkn123}, we obtain that $u\in C(0,T; L^2(\Omega))\cap W^{1,p}(\mathbb{R}^N\times(0,T))$ such that 
$$
\partial_t u-H_p u+(-\Delta_p)^s u\geq 0\text{ in }\Omega,
$$
in the weak sense, then $u\geq 0$ in $\Omega\times(0,T)$ if $u\geq 0$ in $\big(\mathbb{R}^N\setminus\Omega\big)\times(0,T)$ and $u(x,0)\geq 0$ in $\Omega$.
\end{Corollary}

\begin{Corollary}\label{app3}
As an application of Theorem \ref{rmkn12}, we obtain that if $1<p,q<\infty$, $0<s<1$ and $u\in C(0,T; L^2(\Omega))\cap W^{1,p}(\mathbb{R}^N\times(0,T))\cap W^{1,q}(\mathbb{R}^N\times(0,T))$ such that 
$$
\partial_t u-H_p u+(-\Delta_p)^s u-H_q u+(-\Delta_q)^s u\geq 0\text{ in }\Omega,
$$
in the weak sense, then $u\geq 0$ in $\Omega\times(0,T)$ if $u\geq 0$ in $\big(\mathbb{R}^N\setminus\Omega\big)\times(0,T)$ and $u(x,0)\geq 0$ in $\Omega$.
\end{Corollary}

\subsection{Eigenvalue problem}
For $u\in W_0^{1,p}(\Omega)\setminus\{0\}$ and $\lambda\in\mathbb{R}$, we study the following mixed eigenvalue problem
\begin{equation}\label{evpeqn}
-H_p u+(-\Delta_p)^s u=\lambda |u|^{p-2}u\text{ in }\Omega.
\end{equation}
We say that $u\in W_0^{1,p}(\Omega)\setminus\{0\}$ is an eigenfunction of \eqref{evpeqn} corresponding to the eigenvalue $\lambda$, if for every $\phi\in W_0^{1,p}(\Omega)$, we have
\begin{equation}\label{evpwksol}
\begin{split}
&\int_{\Omega}H(\nabla u)^{p-1}\nabla H(\nabla u)\nabla\phi\,dx+\iint_{\mathbb{R}^{2N}}|u(x)-u(y)|^{p-2}(u(x)-u(y))(\phi(x)-\phi(y))\,d\mu\\
&=\lambda\int_{\Omega}|u|^{p-2}u\phi\,dx
\end{split}    
\end{equation}
and we say that $(\lambda,u)$ is an eigenpair of \eqref{evpeqn}.

\begin{Remark}\label{evprmk1}
Then from \cite[Theorem 2.10]{GU}, it holds that $u\in L^\infty(\Omega)$. Further, from \cite[Theorem 2.9]{GU}, the first eigenvalue $\lambda_1=\lambda_1(\Omega)$ of \eqref{evpeqn} is
$$
\lambda_1=\inf\Bigg\{\frac{\int_{\Omega}H(\nabla u)^p\,dx+\iint_{\mathbb{R}^{2N}}|u(x)-u(y)|^p\,d\mu}{\|u\|^p_{L^p(\Omega)}}:u\in W_0^{1,p}(\Omega)\setminus\{0\}\Bigg\}
$$
and the there exists $v\in W_0^{1,p}(\Omega)\setminus\{0\}\cap L^\infty(\Omega)$ such that 
$$
\lambda_1=\frac{\int_{\Omega}H(\nabla v)^p\,dx+\iint_{\mathbb{R}^{2N}}|v(x)-v(y)|^p\,d\mu}{\|v\|^p_{L^p(\Omega)}}.
$$
We observe that the eigenfunction associated to $\lambda_1$ does not change sign. Indeed, if $v$ is an eigenfunction corresponding to $\lambda_1$, so is $|v|$. So, wlog, let $v\geq 0$ in $\Omega$. Since $v\neq 0$, by \cite[Theorem 3.10]{GKK}, for every $\omega\Subset\Omega$, there exists a constant $c_\omega>0$ such that $v\geq c_\omega>0$ in $\omega$. Hence, $v>0$ in $\Omega$.
\end{Remark}
In the local case, the following strict monotonicity property is proved in \cite{AlP, Tyagi}. Here, we prove it in the mixed setting.
\begin{Theorem}\label{monevp}(\textbf{Strict monotonicity of the first eigenvalue})
Let $\Omega_1\subset\Omega_2$ be such that $\Omega_1\neq \Omega_2$. Then $\lambda_1(\Omega_1)>\lambda_1(\Omega_2)$, if both exist.
\end{Theorem}
\begin{proof}
By Remark \ref{evprmk1}, let $u_i\in W_0^{1,p}(\Omega_i)\cap L^\infty(\Omega_i)$ be the positive eigenfunction associated with $\lambda_1(\Omega_i)$ for $i=1,2$, that is
\begin{equation}\label{evpeqn1}
\begin{split}
-H_p u_1+(-\Delta_p)^s u_1&=\lambda_1(\Omega_1)|u_1|^{p-2}u_1\text{ in }\Omega_1\\
u_1&>0\text{ in }\Omega,\quad u_1=0\text{ in }\mathbb{R}^N\setminus\Omega_1
\end{split}
\end{equation}
and
\begin{equation}\label{evpeqn2}
\begin{split}
-H_p u_2+(-\Delta_p)^s u_2&=\lambda_1(\Omega_2)|u_2|^{p-2}u_2\text{ in }\Omega_1\\
u_2&>0\text{ in }\Omega,\quad u_2=0\text{ in }\mathbb{R}^N\setminus\Omega_2.
\end{split}
\end{equation}
Let $\{\phi_n\}_{n\in\mathbb{N}}\subset C_c^{\infty}(\Omega_1)$ be a sequence of nonnegative functions such that $\phi_n\to u_1$ strongly in $W_0^{1,p}(\Omega_1)$. Then by Theorem \ref{BrPI} and Theorem \ref{JacPI}, we have
\begin{equation}\label{evpeqn3}
\begin{split}
0&\leq \int_{\Omega_1}H(\nabla\phi_n)^p\,dx+\iint_{\mathbb{R}^{2N}}|\phi_n(x)-\phi_n(y)|^p\,d\mu-\int_{\Omega_1}H(\nabla u_2)^{p-1}\nabla H(\nabla u_2)\nabla\Big(\frac{\phi_n^{p}}{u_2^{p-1}}\Big)\,dx\\
&\quad-\iint_{\mathbb{R}^{2N}}|u_2(x)-u_2(y)|^{p-2}(u_2(x)-u_2(y))\Big(\frac{\phi_n^{p}}{u_2^{p-1}}(x)-\frac{\phi_n^{p}}{u_2^{p-1}}(y)\Big)\,d\mu.
\end{split}
\end{equation}
Choosing $\frac{\phi_n^p}{u_2^{p-1}}$ as a test function in the weak formulation of \eqref{evpeqn2}, we obtain
\begin{equation}\label{evpeqn4}
\begin{split}
&\int_{\Omega_1}H(\nabla u_2)^{p-1}\nabla H(\nabla u_2)\nabla\Big(\frac{\phi_n^{p}}{u_2^{p-1}}\Big)\,dx\\
&\quad+\iint_{\mathbb{R}^{2N}}|u_2(x)-u_2(y)|^{p-2}(u_2(x)-u_2(y))\Big(\frac{\phi_n^{p}}{u_2^{p-1}}(x)-\frac{\phi_n^{p}}{u_2^{p-1}}(y)\Big)\,d\mu=\lambda_1(\Omega_2)\int_{\Omega_1}\phi_n^{p}\,dx,
\end{split}
\end{equation}
where we have also used that $\phi_n$ are supported inside $\Omega_1$ to make the above integrals over the domain $\Omega_1$. Plugging the estimate \eqref{evpeqn4} in \eqref{evpeqn3}, we deduce that
\begin{equation}\label{evpeqn5}
\begin{split}
0&\leq \int_{\Omega_1}H(\nabla\phi_n)^p\,dx+\iint_{\mathbb{R}^{2N}}|\phi_n(x)-\phi_n(y)|^p\,d\mu-\int_{\Omega_1}H(\nabla u_2)^{p-1}\nabla H(\nabla u_2)\nabla\Big(\frac{\phi_n^{p}}{u_2^{p-1}}\Big)\,dx\\
&\quad-\iint_{\mathbb{R}^{2N}}|u_2(x)-u_2(y)|^{p-2}(u_2(x)-u_2(y))\Big(\frac{\phi_n^{p}}{u_2^{p-1}}(x)-\frac{\phi_n^{p}}{u_2^{p-1}}(y)\Big)\,d\mu\\  
&=\int_{\Omega_1}H(\nabla\phi_n)^p\,dx+\iint_{\mathbb{R}^{2N}}|\phi_n(x)-\phi_n(y)|^p\,d\mu-\lambda_1(\Omega_2)\int_{\Omega_1}\phi_n^{p}\,dx.
\end{split}
\end{equation}
Letting $n\to\infty$ in \eqref{evpeqn5}, we arrive at
\begin{equation}\label{evpeqn6}
\begin{split}
0&\leq\int_{\Omega_1}H(\nabla u_1)^p\,dx+\iint_{\mathbb{R}^{2N}}|u_1(x)-u_1(y)|^p\,d\mu-\int_{\Omega_1}H(\nabla u_2)^{p-1}\nabla H(\nabla u_2)\nabla\Big(\frac{u_1^{p}}{u_2^{p-1}}\Big)\,dx\\
&\quad-\iint_{\mathbb{R}^{2N}}|u_2(x)-u_2(y)|^{p-2}(u_2(x)-u_2(y))\Big(\frac{u_1^{p}}{u_2^{p-1}}(x)-\frac{u_1^{p}}{u_2^{p-1}}(y)\Big)\,d\mu\\
&=\int_{\Omega_1}H(\nabla u_1)^p\,dx+\iint_{\mathbb{R}^{2N}}|u_1(x)-u_1(y)|^p\,d\mu-\lambda_1(\Omega_2)\int_{\Omega_1}u_1^{p}\,dx\\
&=(\lambda_1(\Omega_1)-\lambda_1(\Omega_2))\int_{\Omega_1}u_1^{p}\,dx,
\end{split}
\end{equation}
where the last equality above is deduced by choosing $u_1$ as a test function in the weak formulation of \eqref{evpeqn1}. Therefore, the above inequality yields that
$$
\lambda_1(\Omega_1)-\lambda_1(\Omega_2)\geq 0.
$$
But if $\lambda_1(\Omega_1)-\lambda_1(\Omega_2)=0$, then \eqref{evpeqn6} gives
\begin{equation*}
\begin{split}
&\int_{\Omega_1}H(\nabla u_1)^p\,dx+\iint_{\mathbb{R}^{2N}}|u_1(x)-u_1(y)|^p\,d\mu-\int_{\Omega_1}H(\nabla u_2)^{p-1}\nabla H(\nabla u_2)\nabla\Big(\frac{u_1^{p}}{u_2^{p-1}}\Big)\,dx\\
&\quad-\iint_{\mathbb{R}^{2N}}|u_2(x)-u_2(y)|^{p-2}(u_2(x)-u_2(y))\Big(\frac{u_1^{p}}{u_2^{p-1}}(x)-\frac{u_1^{p}}{u_2^{p-1}}(y)\Big)\,d\mu=0.
\end{split}
\end{equation*}
Applying Theorem \ref{BrPI} and Theorem \ref{JacPI} to the above estimate, we get $u_1=cu_2$ which is not possible, since $\Omega_1\subset\Omega_2$ and $\Omega_1\neq\Omega_2$. This completes the proof.
\end{proof}

The results below in this subsection concerning eigenvalue problem have been established for the operator $-\Delta_p-\Delta_{J,p}$ in \cite{Rossi} when $p>2$, where $\Delta_{J,p}$ is defined in \eqref{djp}. Here, we establish our results for any $p>1$.
\begin{Theorem}\label{evp1}
Let $u\in W_0^{1,p}(\Omega)\setminus\{0\}$ be a given eigenfunction associated to the eigenvalue $\lambda_1$. Let $(\lambda,v)$ be an eigenpair of \eqref{evpeqn} and $v$ be nonegative in $\Omega$ and $\lambda>0$. Then $\lambda=\lambda_1$ and there exists $c\in\mathbb{R}$ such that $v=cu$ in $\mathbb{R}^N$. 
\end{Theorem}

\begin{proof}
Since $v\neq 0$ and $v\geq 0$ in $\Omega$, by the weak Harnack inequality from \cite[Theorem 3.10]{GKK}, we have $v>0$ in $\Omega$. By Remark \ref{evprmk1}, since $u\in W_0^{1,p}(\Omega)\cap L^\infty(\Omega)$ and $u>0$ in $\Omega$, for every $n\in\mathbb{N}$, we have  
$$
w_n=\frac{u^p}{(v+\frac{1}{n})^{p-1}}\in W_0^{1,p}(\Omega).
$$
Therefore, by Theorem \ref{BrPI}, we get
\begin{equation}\label{evp1eqn1}
\begin{split}
0&\leq\int_{\Omega}\big(H(\nabla u)^p-H(\nabla v)^{p-1}\nabla H(\nabla v)\nabla w_n\big)\,dx\\
&=\lambda_1\int_{\Omega}u^p\,dx-\lambda\int_{\Omega}|v|^{p-2}v w_n\,dx\\
&\qquad-\iint_{\mathbb{R}^{2N}}|u(x)-u(y)|^p\,d\mu+\iint_{\mathbb{R}^{2N}}|v(x)-v(y)|^{p-2}(v(x)-v(y))(w_n(x)-w_n(y))\,d\mu\\
&\leq \lambda_1\int_{\Omega}u^p\,dx-\lambda\int_{\Omega}|v|^{p-2}v w_n\,dx,
\end{split}
\end{equation}
where the equality above is obtained by choosing $u$ and $w_n$ as test functions in \eqref{evpwksol} satisfied by $u$ and $v$ respectively. To deduce the last inequality above, we have used the fact that the nonlocal integral 
$$
-\iint_{\mathbb{R}^{2N}}|u(x)-u(y)|^p\,d\mu+\iint_{\mathbb{R}^{2N}}|v(x)-v(y)|^{p-2}(v(x)-v(y))(w_n(x)-w_n(y))\,d\mu\leq 0,
$$
which follows from Theorem \ref{JacPI}. Therefore, using dominated convergence theorem along with Fatou's lemma in \eqref{evp1eqn1}, we arrive at
\begin{align*}
0&\leq \int_{\Omega}\Big(H(\nabla u)^p-H(\nabla v)^{p-1}\nabla H(\nabla v)\nabla\Big(\frac{u^p}{v^{p-1}}\Big)\Big)\,dx\leq (\lambda_1-\lambda)\int_{\Omega}u^p\,dx.
\end{align*}
Hence, $\lambda_1\geq \lambda$. Therefore $\lambda_1=\lambda$, since $\lambda_1$ is the smallest eigenvalue. This gives from the above inequality and Theorem \ref{JacPI} that
$$
H(\nabla u)^p-H(\nabla v)^{p-1}\nabla H(\nabla v)\nabla\Big(\frac{u^p}{v^{p-1}}\Big)=0\text{ in }\Omega.
$$
Therefore, by Theorem \ref{BrPI}, $v=cu$ in $\Omega$ for some constant $c\in\mathbb{R}$.
\end{proof}

\begin{Theorem}\label{evpthm2}
Suppose $(\lambda,u)$ is an eigenpair of \eqref{evpeqn}, where $\lambda>\lambda_1$. Then there exists a positive constant $C$ independent of $u$ such that
$$
|\{u<0\}|>C.
$$
\end{Theorem}
\begin{proof}
By Theorem \ref{evp1}, we get $u_-=\min\{u,0\}\neq 0$. Taking this into account along with the definition of $H$ and Lemma \ref{Happ}, we proceed along the lines of the proof of \cite[Lemma 3.3]{Rossi} to conclude the result.
\end{proof}

\begin{Theorem}\label{evpthm3}
The first eigenvalue $\lambda_1$ is isolated.
\end{Theorem}
\begin{proof}
Let $u_1$ be the eigenfunction associated to $\lambda_1$. Then by Remark \ref{evprmk1}, we have $u_1>0$ in $\Omega$. Taking this into account and Theorem \ref{evpthm2}, the proof follows along the lines of the proof of \cite[Theorem 3.2]{Rossi}.
\end{proof}

\subsection{Weighted Hardy type inequality}
The following Hardy type inequality is proved in the local case in \cite{AlP, BGM}, in the nonlocal case in \cite{Jac}. As far as we are aware, such inequalities are not studied till date for the mixed equation even for the linear operator $-\Delta+(-\Delta)^s$. Here, we establish the below result using the combination of Picone inequalities from Theorem \ref{BrPI} and Theorem \ref{JacPI}.
\begin{Theorem}\label{whi}
Let $v\in W_0^{1,p}(\Omega)$ be a weak supersolution of 
\begin{equation}\label{whieqn1}
-H_p v+(-\Delta_p)^s v\geq \lambda\, g \,v^{p-1}\text{ in }\Omega,\quad v>0\text{ in }\Omega
\end{equation}
as in Definition \ref{p1sol}, where $\lambda>0$ and $g$ is a nonnegative continuous function in $\Omega$. Then for every nonnegative $u\in W_0^{1,p}(\Omega)$, we have
$$
\lambda\int_{\Omega}g u^p\,dx\leq \int_{\Omega}H(\nabla u)^p\,dx+\iint_{\mathbb{R}^{2N}}{|u(x)-u(y)|^p}\,d\mu.
$$
\end{Theorem}
\begin{proof}
By density, we assume that there exists a sequence $\{\phi_n\}_{n\in\mathbb{N}}\subset C_0^{\infty}(\Omega)$, $\phi_n>0$ such that $\phi_n\to u$ strongly in $W_0^{1,p}(\Omega)$. We define $v_\epsilon=v+\epsilon$ for $\epsilon>0$ and $\psi_{n,\epsilon}=\frac{\phi_n^{p}}{(v+\epsilon)^{p-1}}$ for $\epsilon>0$ and $n\in\mathbb{N}$. Then we observe that $\psi_{n,\epsilon}\in W_0^{1,p}(\Omega)$. By Theorem \ref{BrPI} and Theorem \ref{JacPI}, we obtain
\begin{equation}\label{whieqn2}
\begin{split}
0&\leq \int_{\Omega}H(\nabla\phi_n)^p\,dx-\int_{\Omega}H(\nabla v_\epsilon)^{p-1}\nabla H(\nabla v_\epsilon)\nabla \psi_{n,\epsilon}\,dx\\
&\quad +\iint_{\mathbb{R}^{2N}}{|\phi_n(x)-\phi_n(y)|^p}\,d\mu-\iint_{\mathbb{R}^{2N}}|v_\epsilon(x)-v_\epsilon(y)|^{p-2}(v_\epsilon(x)-v_\epsilon(y))(\psi_{n,\epsilon}(x)-\psi_{n,\epsilon}(y))\, d\mu\\
&\leq \int_{\Omega}H(\nabla\phi_n)^p\,dx+\iint_{\mathbb{R}^{2N}}{|\phi_n(x)-\phi_n(y)|^p}\,d\mu-\lambda\int_{\Omega}g\Big(\frac{v}{v_\epsilon}\Big)^{p-1}\phi_n^{p}\,dx.
\end{split}
\end{equation}
The last inequality above is found by choosing $\psi_{n,\epsilon}$ as test function in the weak formulation of \eqref{whieqn1}. Passing $\epsilon\to 0^+$ and using Fatou's lemma, we have
$$
0\leq \int_{\Omega}H(\nabla\phi_n)^p\,dx+\iint_{\mathbb{R}^{2N}}{|\phi_n(x)-\phi_n(y)|^p}\,d\mu-\lambda\int_{\Omega}g\phi_n^{p}\,dx.
$$
Next, passing $n\to\infty$ in the above inequality, the result follows.
\end{proof}
\subsection{System of singular nonlinear equations}
The following type of result is established in local singular system in \cite{Balejde}, and in the nonlocal singular system in \cite{Jac}. Although it is worth mentioning that such results have not been studied till date for the mixed singular system, even for the linear operator $-\Delta +(-\Delta)^s$. Below, we prove a linear relationship between the solutions of the system of equations. For the study of mixed singular equations, see \cite{ArRadu, GKK, GUna} and the references therein. 
\begin{Theorem}\label{sne}
Let $(u,v)\in \Big(W_0^{1,p}(\Omega)\cap L^\infty(\Omega)\times W_0^{1,p}(\Omega)\Big)$ be weak solutions of 
\begin{equation}\label{sneeqn1}
-H_p u+(-\Delta_p)^s u=f(x)v^{p-1}\text{ in }\Omega,\quad u>0\text{ in }\Omega,
\end{equation}
and
\begin{equation}\label{sneeqn2}
-H_p v+(-\Delta_p)^s v=f(x)\Big(\frac{v^2}{u}\Big)^{p-1}\text{ in }\Omega,\quad v>0\text{ in }\Omega,
\end{equation}
as in Definition \ref{p1sol}, where $0\lneqq f(x)\in L^\infty(\Omega)$. Then there exists a constant $k>0$ such that $v=ku$ in $\Omega$.
\end{Theorem}

\begin{proof}
Let $\epsilon>0$ and we set $v_\epsilon=v+\epsilon$. Choosing $\phi=u$ and $\psi_\epsilon=\frac{u^p}{v_\epsilon^{p-1}}$ as test functions in the weak formulation of \eqref{sneeqn1} and \eqref{sneeqn2} respectively and subtracting the resulting equations, we obtain
\begin{equation}\label{sneeqn3}
\begin{split}
&\int_{\Omega}H(\nabla u)^p\,dx-\int_{\Omega}H(\nabla v_\epsilon)^{p-1}\nabla H(\nabla v_\epsilon)\nabla\psi_\epsilon\,dx\\
&+\iint_{\mathbb{R}^{2N}}{|u(x)-u(y)|^p}\,d\mu-\iint_{\mathbb{R}^{2N}}|v_\epsilon(x)-v_\epsilon(y)|^{p-2}(v_\epsilon(x)-v_\epsilon(y))(\psi_\epsilon(x)-\psi_\epsilon(y))\,d\mu\\
&=\int_{\Omega}uf(x)\Big(v^{p-1}-\Big(\frac{v^2}{v_\epsilon}\Big)^{p-1}\Big)\,dx.
\end{split}
\end{equation}
Passing $\epsilon\to 0^+$ and using Fatou's lemma, along with Theorem \ref{BrPI} and Theorem \ref{JacPI}, we obtain
\begin{equation}\label{sneeqn4}
\begin{split}
0&=\int_{\Omega}H(\nabla u)^p\,dx-\int_{\Omega}H(\nabla v)^{p-1}\nabla H(\nabla v)\nabla\Big(\frac{u^p}{v^{p-1}}\Big)\,dx+\iint_{\mathbb{R}^{2N}}{|u(x)-u(y)|^p}\,d\mu\\
&\quad-\iint_{\mathbb{R}^{2N}}|v(x)-v(y)|^{p-2}(v(x)-v(y))\Big(\Big(\frac{u^p}{v^{p-1}}\Big)(x)-\Big(\frac{u^p}{v^{p-1}}\Big)(y)\Big)\,d\mu.
\end{split}
\end{equation}
Again using Theorem \ref{BrPI} and Theorem \ref{JacPI}, we obtain from \eqref{sneeqn4} that
$$
\int_{\Omega}H(\nabla u)^p\,dx-\int_{\Omega}H(\nabla v)^{p-1}\nabla H(\nabla v)\nabla\Big(\frac{u^p}{v^{p-1}}\Big)\,dx=0,
$$
which gives by Theorem \ref{BrPI} that $v=ku$ for some constant $k>0$ as desired.
\end{proof}
\subsection{Non existence of positive super-solutions}
The following nonexistence result is studied in the $p$-Laplace case in \cite{AlP, Tyagi}, weighted $p$-Laplace setting in \cite{Garainopmath}. As far as we are aware, such results have not been studied till date in the mixed case, even for the linear operator $-\Delta+(-\Delta)^s$. The proof follows by choosing suitable test function along with Picone's inequality.
\begin{Theorem}\label{np}
If there exists $u\in W_0^{1,p}(\Omega)\cap L^\infty(\Omega)$, $u\geq 0$ in $\Omega$ such that $$J(u)=\int_{\Omega}H(\nabla u)^p dx+\iint_{\mathbb{R}^{2N}}|u(x)-u(y)|^p\,d\mu-\int_{\Omega}g_1(x) u^p dx<0,$$ then the problem 
\begin{equation}\label{np1}
-H_p v+(-\Delta_p)^s v-g_1(x)|v|^{p-2}v=g_2(x)\text{ in }\Omega,
\end{equation}
where $0\leq{g_1}\in{L^{\infty}(\Omega)}$ and $0\leq g_2\in L^q(\Omega)$ with $q=\frac{p}{p-1}$ has no positive weak supersolution in $W_0^{1,p}(\Omega)$ as in Definition \ref{p1sol}.
\end{Theorem}

\begin{proof}
Suppose $v$ is a positive weak supersolution of (\ref{np1}). Let $\epsilon>0$, then choosing $\varphi=\frac{u^p}{(v+\epsilon)^{p-1}}\in{W_0^{1,p}(\Omega)}$ as a test function in the weak formulation of \eqref{np1}, we get
\begin{align*}
&\int_{\Omega}g_1(x)\Big(\frac{v}{v+\epsilon}\Big)^{p-1}u^p\,dx\\
&\leq \int_{\Omega}H(\nabla v)^{p-1}\nabla H(\nabla v)\nabla (\frac{u^p}{(v+\epsilon)^{p-1}})\,dx\\
&\quad+\iint_{\mathbb{R}^{2N}}|v(x)-v(y)|^{p-2}(v(x)-v(y))(\phi_\epsilon(x)-\phi_\epsilon(y))\,d\mu-\int_{\Omega}g_2(x)\frac{u^p}{(v+\epsilon)^{p-1}}dx\\
&=\int_{\Omega}H(\nabla u)^p\,dx-\int_{\Omega}H(\nabla u)^p\,dx+\int_{\Omega}H(\nabla v)^{p-1}\nabla H(\nabla v)\nabla (\frac{u^p}{(v+\epsilon)^{p-1}})\,dx\\
&\quad+\iint_{\mathbb{R}^{2N}}|u(x)-u(y)|^p\,d\mu-\iint_{\mathbb{R}^{2N}}|u(x)-u(y)|^p\,d\mu\\
&\quad+\iint_{\mathbb{R}^{2N}}|v(x)-v(y)|^{p-2}(v(x)-v(y))(\phi_\epsilon(x)-\phi_\epsilon(y))\,d\mu-\int_{\Omega}g_2(x)\frac{u^p}{(v+\epsilon)^{p-1}}dx\\
&=\int_{\Omega}\Big[H(\nabla u)^p\,dx+\iint_{\mathbb{R}^{2N}}|u(x)-u(y)|^p\,d\mu\\
&\quad-\int_{\Omega}R(u,v+\epsilon)\,dx-\int_{\Omega}S(u,v+\epsilon)\,d\mu-\int_{\Omega}g_2(x)\frac{u^p}{(v+\epsilon)^{p-1}}\Big]dx\\
&\leq \int_{\Omega}\Big[H(\nabla u)^p\,dx+\iint_{\mathbb{R}^{2N}}|u(x)-u(y)|^p\,d\mu, 
\end{align*}
since
$$
R(u,v+\epsilon)=H(\nabla u)^p-H(\nabla v)^{p-1}\nabla H(\nabla v)\nabla (\frac{u^p}{(v+\epsilon)^{p-1}})\geq 0
$$
and
$$
S(u,v+\epsilon)=|u(x)-u(y)|^p-|v(x)-v(y)|^{p-2}(v(x)-v(y))(\phi_\epsilon(x)-\phi_\epsilon(y))\geq 0
$$
by Theorem \ref{BrPI} and Theorem \ref{JacPI} respectively. Then Fatou's lemma yields $J(u)\geq 0$, which is a contradiction. Hence the theorem follows.
\end{proof}

\section{Part-II}
\subsection{Elliptic case}
\subsubsection{Caccioppoli type energy estimates}
In this subsection, we establish Caccioppoli type energy estimates for weak subsolutions and weak supersolutions of the following equation
\begin{equation}\label{meeqn}
-H_p u+\epsilon(-\Delta_p)^s u=g(x)|u|^{p-2}u\text{ in }\Omega,
\end{equation}
where $\epsilon\in(0,1]$, $g\in L^\infty_{\mathrm{loc}}(\Omega)$. In the anisotropic case, such results can be found in \cite{Jaros}, where Picone identity played a crucial role and in the nonlocal setting, we refer to \cite{DKPhar}. In the mixed case, see \cite{GKK, GK}. Here, we show that subsolution estimates can be derived using the Picone identity from Theorem \ref{BrPI}. Before defining the notion of solutions, we need the following tail space from \cite{DKPhar, DKPloc} defined by
$$
L^{ps}_{p-1}(\mathbb{R}^N)=\left\{u\in L^{p-1}_{\mathrm{loc}}(\mathbb{R}^N):\int_{\mathbb{R}^N}\frac{|u(y|^{p-1}}{(1+|y|)^{N+ps}}\, dy<\infty\right\}.
$$
It is easy to check that $W^{s,p}(\mathbb{R}^N)$ and $L^\infty(\mathbb{R}^N)$ are contained in the tail space $L^{ps}_{p-1}(\mathbb{R}^N)$. The mixed local and nonlocal tail denoted by $\mathrm{Tail}\,(u;x_0,r)$ for any $x_0\in\mathbb{R}^N$ and $r>0$ is well defined if $u\in L^{ps}_{p-1}(\mathbb{R}^N)$, where
\begin{equation}\label{taildef}
\mathrm{Tail}\,(u;x_0,r):=\left(r^p\int_{\mathbb{R}^N\setminus B_r(x_0)}\frac{|u(y|^{p-1}}{|y-x_0|^{N+ps}}\, dy\right)^\frac{1}{p-1}.
\end{equation}
For more details, we refer to \cite{GK} and the references therein.
\begin{Definition}\label{wheqnwksol}
Let $g\in L_\mathrm{loc}^\infty(\Omega)$. We say that $u\in W^{1,p}_{\mathrm{loc}}(\Omega)\cap L^{ps}_{p-1}(\mathbb{R}^N)$ is a weak supersolution (or subsolution) of \eqref{meeqn}, if for every $\Omega'\Subset\Omega$ and every nonnegative $\phi\in W_0^{1,p}(\Omega')$, we have
\begin{equation}\label{wheqnwksol1}
\begin{split}
&\int_{\Omega'}H(\nabla u)^{p-1}\nabla H(\nabla u)\nabla\phi\,dx\\
&\quad+\epsilon\iint_{\mathbb{R}^{2N}}|u(x)-u(y)|^{p-2}(u(x)-u(y))(\phi(x)-\phi(y))\,d\mu-\int_{\Omega'}g\,|u|^{p-2}u\phi\,dx\geq (\text{ or }\,\leq)\, 0.
\end{split}
\end{equation}
Further, we say that $u\in W^{1,p}_{\mathrm{loc}}(\Omega)\cap L^{ps}_{p-1}(\mathbb{R}^N)$ is a weak solution of \eqref{meeqn}, if the L.H.S. integral in \eqref{wheqnwksol1} become equal to zero for every $\Omega'\Subset\Omega$ and every $\phi\in W_0^{1,p}(\Omega')$ without any sign restriction.
\end{Definition}
We remark that Definition \ref{wheqnwksol} is well stated by Lemma \ref{l1} and Lemma \ref{locnon1}. First, we establish the following energy estimates for weak supersolutions.
\begin{Theorem}\label{engsup}(\textbf{Caccioppoli type inequality for supersolutions})
Let $g\in L^\infty_{\mathrm{loc}}(\Omega)$ and suppose $\Omega'\Subset\Omega$. Let $v\in W^{1,p}_{\mathrm{loc}}(\Omega)\cap L^\infty(\mathbb{R}^N)$ be a weak supersolution of \eqref{meeqn} in $\Omega$ as in Definition \ref{wheqnwksol} such that there exists a constant $\rho>0$ with $v\geq\rho>0$ in $\mathbb{R}^N$. Then for every $q\in(1,p)$, there exists a constant $C=C(p)>0$ such that
\begin{equation}\label{engsupest}
\begin{split}
\int_{\Omega'}v^{-q}H(\nabla v)^p\psi^p\,dx
&\leq \frac{C}{(q-1)^\frac{p}{p-1}}\int_{\Omega'}v^{p-q}|\nabla\psi|^p\,dx-\frac{2}{q-1}\int_{\Omega'}g(x)v(x)^{p-q}\psi(x)^p\,dx\\
&+\frac{1}{(q-1)}\frac{C}{\big(\zeta(q)+1+\frac{1}{(q-1)^{p-1}}\big)}\iint_{\Omega'\times \Omega'}\max\{v(x)^{p-q},v(y)^{p-q}\}|\psi(x)-\psi(y)|^p\,d\mu\\
&\quad+\frac{C}{q-1}\left(\sup_{x\in\mathrm{spt}\,\psi}\int_{\mathbb{R}^N\setminus\Omega'}\frac{dy}{|y-x|^{N+ps}}\,dy\right)\int_{\Omega'}u^{p-q}\psi^p\,dx,
\end{split}    
\end{equation}
for every nonnegative function $\psi\in C_c^\infty(\Omega')$, where
$$
\zeta(q)=
\begin{cases}
    \frac{(q-1)p^p}{p-q}\quad \text{ if }0<p-q<1,\\
    \frac{(q-1)p^p}{(p-q)^p}\quad \text{ otherwise}.
\end{cases}
$$
\end{Theorem}

\begin{proof}
 Since $v$ is a weak supersolution of \eqref{meeqn}, using that $v\geq\rho>0$ in $\mathbb{R}^N$, we test \eqref{wheqnwksol1} with $\phi=v^{1-q}\psi^p$ and obtain
\begin{equation}\label{supwksolap1}
\begin{split}
    \int_{\Omega'}g\,v^{p-q}\psi^p\,dx&\leq I+J,
\end{split}
\end{equation}
where
$$
I=\int_{\Omega'}H(\nabla v)^{p-1}\nabla H(\nabla v)\nabla (v^{1-q}\psi^p)\,dx,
$$

$$
J=\epsilon\iint_{\mathbb{R}^{2N}}|v(x)-v(y)|^{p-2}(v(x)-v(y))(v(x)^{1-q}\psi(x)^p-v(y)^{1-q}\psi(y)^p)\,d\mu.
$$
\textbf{Estimate of $I$:} Using Lemma \ref{Happ} along with Cauchy-Schwartz and Young's inequality, we obtain
\begin{equation}\label{est1}
I\leq (1-q)I_1+I_2,
\end{equation}
where
$$
I_1=\int_{\Omega'}v^{-q}\psi^p H(\nabla v)^p\,dx,
$$
and
\begin{equation}\label{estI2}
\begin{split}
I_2&=C\int_{\Omega'}v^{1-q}|\nabla\psi|H(\nabla v)^{p-1}\psi^{p-1}\,dx\leq\frac{q-1}{2}I_1+\frac{C}{(q-1)^\frac{1}{p-1}}\int_{\Omega'}v^{p-q}|\nabla\psi|^p\,dx,
\end{split}
\end{equation}
for some constant $C=C(p)>0$.\\
\textbf{Estimate of $J$:} Using Lemma \ref{Inequality1}, we obtain
\begin{equation}\label{estJ}
\begin{split}
J&\leq \big(\zeta(q)+1+\frac{1}{(q-1)^{p-1}}\big)  \iint_{\Omega'\times\Omega'}\max\{v(x)^{p-q},v(y)^{p-q}\}|\psi(x)-\psi(y)|^p\,d\mu\\
&\quad +C\left(\sup_{x\in\mathrm{spt}\,\psi}\int_{\mathbb{R}^N\setminus\Omega'}\frac{dy}{|x-y|^{N+ps}}\right)\int_{\Omega'}v^{p-q}\psi^p\,dx,
\end{split}
\end{equation}
where
$$
\zeta(q)=
\begin{cases}
    \frac{(q-1)p^p}{p-q}\quad \text{ if }0<p-q<1,\\
    \frac{(q-1)p^p}{(p-q)^p}\quad \text{ otherwise}.
\end{cases}
$$
Combining the estimates \eqref{est1}, \eqref{estI2} and \eqref{estJ} into \eqref{supwksolap1}, the result follows.
\end{proof}
The following logarithmic estimate follows similarly as in the proof of Theorem \ref{engsup} above except that here we choose the test function $\phi=v^{1-p}\psi^p$. 
\begin{Theorem}\label{englog}(\textbf{Logarithmic Caccioppoli type inequality for supersolutions})
Let $g\in L^\infty_{\mathrm{loc}}(\Omega)$ and suppose $\Omega'\Subset\Omega$. Let $v\in W^{1,p}_{\mathrm{loc}}(\Omega)\cap L^\infty(\mathbb{R}^N)$ be a weak supersolution of \eqref{meeqn} in $\Omega$ as in Definition \ref{wheqnwksol} such that there exists a constant $\rho>0$ with $v\geq\rho>0$ in $\mathbb{R}^N$. Then there exists a constant $C=C(p)>0$ such that
\begin{equation}\label{logest}
\begin{split}
\int_{\Omega'}\psi^p\,H(\nabla \mathrm{log}\,v)^p\,dx&\leq C\int_{\Omega'}|\psi(x)-\psi(y)|^p\,d\mu+C\int_{x\in\mathrm{spt}\,\psi}\int_{\mathbb{R}^N\setminus\Omega'}\psi(x)^p\,d\mu\\
&\quad+C\int_{\Omega'}|\nabla\psi|^p\,dx-\int_{\Omega'}g\psi^p\,dx,
\end{split}
\end{equation}
for every nonnegative function $\psi\in C_c^\infty(\Omega')$. 
\end{Theorem}
Finally, we obtain the following energy estimate for weak subsolutions, where Picone identity is crucially used.
\begin{Theorem}\label{engsub}(\textbf{Caccioppoli type inequality for subsolutions})
Let $g\in L^\infty_{\mathrm{loc}}(\Omega)$ and suppose $\Omega'\Subset\Omega$. Let $u\in W^{1,p}_{\mathrm{loc}}(\Omega)\cap L^\infty(\mathbb{R}^N)$ be a weak subsolution of \eqref{meeqn} in $\Omega$ as in Definition \ref{wheqnwksol} such that $u>0$ in $\Omega'$. Then for every $q>p$, it holds that
\begin{equation}\label{engsubest}
\begin{split}
\int_{\Omega'}u^{q-p}H(\eta \nabla u)^p\,dx&\leq \Big(\frac{p}{q-p+1}\Big)^p\int_{\Omega'}u^q H(\nabla\eta)^p\,dx+\frac{p}{q-p+1}\int_{\Omega'}g(x)u^q\eta^p\,dx\\
&\quad+\frac{C\,p}{q-p+1}\int_{\Omega'}\int_{\Omega'}(u(x)^q+u(y)^q)|\eta(x)-\eta(y)|^p\,d\mu\\
&\quad+\frac{2p}{q-p+1}\left(\sup_{x\in\mathrm{spt}\,\eta}\int_{\mathbb{R}^N\setminus\Omega'}\frac{|u(y)|^{p-1}}{|x-y|^{N+ps}}\,dy\right)\int_{\Omega'}u^{q-p+1}(x)\eta(x)^p\,dx,
\end{split}
\end{equation}
for every nonnegative $\eta\in C_{c}^\infty(\Omega')$ for some constant $C=C(p)>0$.
\end{Theorem}

\begin{proof}
Let $v=u^\frac{q}{p}\eta$ and choosing $\phi=\frac{v^p}{u^{p-1}}\in W_0^{1,p}(\Omega')$ as a test function in \eqref{wheqnwksol1}, we get
\begin{equation}\label{mewksolap1}
\begin{split}
\int_{\Omega'}H(\nabla v)^p\,dx-\int_{\Omega'}g(x)v^p\,dx&\leq I-J,
\end{split}
\end{equation}
where
$$
I=\int_{\Omega'}H(\nabla v)^p\,dx-\int_{\Omega'}H(\nabla u)^{p-1}\nabla H(\nabla u)\nabla \Big(\frac{v^p}{u^{p-1}}\Big)\,dx
$$
and
$$
J=\iint_{\mathbb{R}^{2N}}|u(x)-u(y)|^{p-2}(u(x)-u(y))\Big(\frac{v^p}{u^{p-1}}(x)-\frac{v^p}{u^{p-1}}(y)\Big)\, d\mu.
$$
Since $v=u^\frac{q}{p}\eta$, proceeding along the lines of the proof of \cite[Prop 3.5, pages 13-16]{BrPr}, we obtain
\begin{equation}\label{BrParap}
\begin{split}
-J&\leq C\,\epsilon\int_{\Omega'}\int_{\Omega'}(u(x)^q+u(y)^q)|\eta(x)-\eta(y)|^p\,d\mu\\
&\quad+2\,\epsilon\left(\sup_{x\in\mathrm{spt}\,\eta}\int_{{\mathbb{R}^N}\setminus\Omega'}\frac{|u(y)|^{p-1}}{|x-y|^{N+ps}}\,dy\right)\int_{\Omega'}u(x)^{q-p+1}\eta(x)^p\,dx,
\end{split}    
\end{equation}
for some constant $C=C(p)>0$. Using $v=u^\frac{q}{p}\eta$ and Picone's identity from Theorem \ref{BrPI}, we follow the proof of \cite[Theorem 3.1, page 1142]{Jaros} to estimate the integral $I$ and combine it with the above estimate of $J$ and \eqref{mewksolap1} to obtain the required estimate \eqref{engsubest}.
\end{proof}
\subsubsection{Weak Harnack inequality}
In this subsection, we establish weak Harnack inequality for weak supersolutions of the equation \eqref{meeqn} given by
\begin{equation}\label{wheqn}
-H_p u+\epsilon(-\Delta_p)^s u=g(x)|u|^{p-1}u\text{ in }\Omega
\end{equation}
where $g\in L_\mathrm{loc}^\infty(\Omega)$ and $\epsilon\in(0,1]$.\\

When $g\equiv 0$ and $\epsilon=1$, higher H\'older regularity is obtained in \cite{Min, GLcvpde} assuming that $H(x)=|x|$. Further, when $g\equiv 0$ and $\epsilon=1$, weak Harnack inequality for \eqref{wheqn} is proved in \cite{GK} for $H(x)=|x|$ and extended to any Finsler-Minkowski norm $H$ in \cite{GKK} for any $1<p<\infty$. Further, recently, to study a related singular critical problem, when $p=2$, $\epsilon\in(0,1]$, $0<s<1$ and $H(x)=|x|$, weak Harnack inequality is established in \cite[Proposition 3.3]{Biagi}.  In this article, below we prove weak Harnack inequality for the equation \eqref{wheqn} for any $1<p<\infty$ which extends \cite[Proposition 3.3]{Biagi} to the anisotropic and quasilinear setting. The main result of this subsection is stated as follows, whose proof is similar to the proof of \cite[Lemma 3.7]{GKK} and \cite[Proposition 3.3]{Biagi}. Indeed, the proof follows from Lemma \ref{expthm}
and proceeding along the lines of the proof of \cite[Lemma 4.1]{DKPhar}.
\begin{Theorem}\label{whthm}(\textbf{Weak Harnack inequality})
Let $g\in L^\infty(\Omega)$ be nonpositive in $\Omega$ and suppose that $u\in W^{1,p}_{\mathrm{loc}}(\Omega)\cap L^{ps}_{p-1}(\mathbb{R}^N)$ is a weak supersolution of \eqref{wheqn} such that $u\geq 0$ in $B_R(x_0)\Subset\Omega$. Then there exist constants $Q=Q(N,s,p,g,C_1,C_2)>0$ and $C=C(N,s,p,g,C_1,C_2)>0$ which are independent of $\epsilon\in(0,1]$ such that
\begin{equation}\label{whine}
\Big(\fint_{B_r(x_0)}u^Q\,dx\Big)^\frac{1}{Q}\leq C\inf_{B_r(x_0)}\,u+C\Big(\frac{r}{R}\Big)^\frac{p}{p-1}\mathrm{Tail}\,(u_-;x_0,R),
\end{equation}
where $B_r(x_0)\subset B_R(x_0)$ and $r\in(0,1]$.
\end{Theorem}
To prove Theorem \ref{whthm}, first we establish the following energy estimate.
\begin{Lemma}\label{wheng}(\textbf{Caccioppoli type energy estimates for sub and supersolutions})
Let $g\in L^\infty(\Omega)$ and $u\in W^{1,p}_{\mathrm{loc}}(\Omega)\cap L^{ps}_{p-1}(\mathbb{R}^N)$ be a weak subsolution (resp. supersolution) of \eqref{wheqn}. We arbitrarily fix $k\in\mathbb{R}$ and set $w=(u-k)_+$ [resp. $w=(u-k)_-$]. Then there exists a constant $C=C(p,C_1,C_2)>0$ such that
\begin{equation}\label{wheng1}
\begin{split}
&\int_{B_r(x_0)}\psi^p|\nabla w|^p\,dx+\epsilon\iint_{B_r(x_0)\times B_r(x_0)}|w(x)\psi(x)-w(y)\psi(y)|^p\,d\mu\\
&\quad-[\text{resp.}\,+]\int_{B_r(x_0)}g(x)\,|u|^{p-2}u\,w\,\psi^p\,dx\\
&\leq C\Big(\int_{B_r(x_0)}w^p|\nabla\psi|^p\,dx+\iint_{B_r(x_0)\times B_r(x_0)}\max\{w(x),w(y)\}^p|\psi(x)-\psi(y)|^p\,d\mu\\
&\quad+\sup_{x\in\mathrm{spt}\,\psi}\int_{\mathbb{R}^N\setminus B_r(x_0)}\frac{w(y)^{p-1}}{|x-y|^{N+ps}}\, dy\int_{B_r(x_0)}w\,\psi^p\,dx\Big)
\end{split}
\end{equation}
whenever $B_r(x_0)\Subset\Omega$ and $\psi\in C_c^{\infty}(B_r(x_0))$ is nonnegative.
\end{Lemma}
\begin{proof}
Let $u$ be a weak subsolution of \eqref{wheqn}. Then choosing $\phi=w\,\psi^p=(u-k)_+\,\psi^p$ as a test function in \eqref{wheqnwksol1}, we obtain
\begin{equation}\label{wheqnwksol2}
\begin{split}
0&\geq I+J-\int_{B_r(x_0)}g(x)|u|^{p-2}uw\,\psi^p\,dx,
\end{split}
\end{equation}
where
$$
I=\int_{B_r(x_0)}H(\nabla u)^{p-1}\nabla H(\nabla u)\nabla (w\psi^p)\,dx,
$$
$$
J=\epsilon\iint_{\mathbb{R}^{2N}}|u(x)-u(y)|^{p-2}(u(x)-u(y))(w(x)\psi(x)^p-w(y)\psi(y)^p)\,d\mu.
$$
From the estimate (3.3) of $I$ in the proof of \cite[Lemma 3.1]{GKK}, we have
\begin{equation}\label{whestI}
\begin{split}
I&\geq c\int_{B_r(x_0)}\psi^p|\nabla w|^p\,dx-C\int_{B_r(x_0)}w^p|\nabla\psi|^p\,dx,
\end{split}
\end{equation}
for some positive constant $C=C(p,C_1,C_2)$. Moreover, from the lines of the proof of \cite[Pages 1285–1287, Theorem 1.4]{DKPloc}, for some constants $c = c(p) > 0$ and $C = C(p) > 0$, we have
\begin{equation}\label{whestJ}
\begin{split}
J&\geq c\,\epsilon\iint_{B_r(x_0)\times B_r(x_0)}|\psi(x)w(x)-\psi(y)w(y)|^p\,d\mu\\
&\quad -C\epsilon\iint_{B_r(x_0)\times B_r(x_0)}\max\{w(x),w(y)\}^p|\psi(x)-\psi(y)|^p\,d\mu\\
&\quad -C\epsilon\sup_{x\in\mathrm{spt}\,\psi}\int_{\mathbb{R}^N\setminus B_r(x_0)}\frac{w(y)^{p-1}}{|x-y|^{N+ps}}\int_{B_r(x_0)}w\,\psi^p\,dx.
\end{split}
\end{equation}
Thus, combining the estimates \eqref{whestI}, \eqref{whestJ} in \eqref{wheqnwksol2} and recalling that $\epsilon\in(0,1]$, the resulting estimate \eqref{wheng1} follows. Finally, the estimate for supersolutions follows by observing that $-u$ is a subsolution of \eqref{wheqn} if $u$ is a supersolution of \eqref{wheqn} and therefore by applying the obtained estimate \eqref{wheng1} to $-u$.
\end{proof}

Next, we show that expansion of positivity property holds for weak superslutions of the problem \eqref{wheqn}. The proof follows from the proof of \cite[Lemma 3.5]{Biagi} and \cite[Lemma 3.6]{GKK}, which in turn is similar to \cite[Lemma 7.1]{GK}. For convenience of the reader, we provide few details of the proof.

\begin{Lemma}\label{expthm}(\textbf{Expansion of positivity})
Let $g\in L^\infty(\Omega)$ be nonpositive in $\Omega$ and $u\in W^{1,p}_{\mathrm{loc}}(\Omega)\cap L^{ps}_{p-1}(\mathbb{R}^N)$ be a weak supersolution of \eqref{wheqn} such that $u\geq 0$ in $B_R(x_0)\Subset\Omega$. Let $k\geq 0$ and suppose that there exists a constant $\tau\in(0,1]$ such that
\begin{equation}\label{expgc}
|B_r(x_0)\cap \{u\geq k\}|\geq\tau|B_r(x_0)|
\end{equation}
for some $r\in(0,1]$ such that $0<r<\frac{R}{16}$. Then there exists a constant $\delta=\delta(N,p,s,\tau,g,C_1,C_2)\in(0,\frac{1}{4})$ such that
\begin{equation}\label{expc}
\inf_{B_{4r}(x_0)}\geq \delta\,k-\Big(\frac{r}{R}\Big)^\frac{p}{p-1}\mathrm{Tail}\,(u_-;x_0,R)
\end{equation}
where $\mathrm{Tail}(\cdot)$ is given by \eqref{taildef}.
\end{Lemma}
\begin{proof}
We prove the result into two steps.\\
\textbf{Step 1.} Let \eqref{expgc} holds. Then we claim that there exists a constant $c_1=c_1(N,p,s,g,C_1,C_2)>0$ such that
\begin{equation}\label{exp1}
\left|B_{6r}(x_0)\cap \Big\{u\leq 2\delta\,k-\frac{1}{2}\Big(\frac{r}{R}\Big)^\frac{p}{p-1}\mathrm{Tail}\,(u_-;x_0,R)-\eta\Big\}\right|\leq \frac{c_1}{\tau\,\mathrm{log}\,\frac{1}{2\delta}}|B_{6r}(x_0)|
\end{equation}
for every $\delta\in(0,\frac{1}{4})$ and every $\eta>0$.\\
Let $\eta>0$ and $\psi\in C_c^{\infty}(B_{7r}(x_0))$ be a cutoff function such that 
$0\leq\psi\leq 1$ in $B_{7r}(x_0)$, $\psi\equiv 1$ in $B_{6r}(x_0)$ and $|\nabla\psi|\leq\frac{8}{r}$ in $B_{7r}(x_0)$.
We denote $w=u+t_{\eta}$, where
$$
t_{\eta}=\frac{1}{2}\Big(\frac{r}{R}\Big)^\frac{p}{p-1}\Tail(u_{-};x_0,R)+\eta.
$$
We observe that $w$ is a weak supersolution of \eqref{wheqn}, since $u$ is so and $a\geq 0$ along with $u\geq 0$ in $B_R(x_0)$. Therefore, choosing $\phi=w^{1-p}\psi^p$ as a test function in \eqref{wheqnwksol1}, we obtain
\begin{equation}\label{DGLtsteqn1}
\begin{split}
0&\leq\int_{B_{8r}(x_0)}H(\nabla w)^{p-1}\nabla H(\nabla w)\nabla( w^{1-p}\psi^p)\,dx\\
&\qquad+\iint_{B_{8r}(x_0)\times B_{8r}(x_0)}|w(x)-w(y)|^{p-2}(w(x)-w(y))(w(x)^{1-p}\psi(x)^p- w(y)^{1-p}\psi(y)^p)\,d\mu\\
&\qquad+2\int_{\mathbb{R}^n\setminus B_{8r}(x_0)}\int_{B_{8r}(x_0)}|w(x)-w(y)|^{p-2}(w(x)-w(y)) w(x)^{1-p}\psi(x)^p\,d\mu-\int_{B_{8r}(x_0)}g(x)\,\psi(x)^p\,dx\\
&=I_1+I_2+I_3+I_4.
\end{split}
\end{equation}
\textbf{Estimate of $I_1$:} Proceeding similarly as in the proof of \cite[Pages 717--718, Lemma 3.4]{KKma}, we obtain a constant $c=c(p,C_1,C_2)>0$ such that
\begin{equation}\label{EstI1}
\begin{split}
I_1&\leq -c\int_{B_{6r}(x_0)}|\nabla\log w|^p\,dx+cr^{N-p}.
\end{split}
\end{equation}
\textbf{Estimate of $I_2$:} Arguing as in the proof of the estimate of $I_1$ in \cite[page 1817]{DKPhar} and using the fact that $r\in(0,1]$ and the properties of $H$, we
obtain a constant $c=c(N,p,s)>0$ such that
\begin{equation}\label{EstI2}
\begin{split}
I_2&\leq -\frac{1}{c}\iint_{B_{6r}(x_0)\times B_{6r}(x_0)}\bigg|\log\Big(\frac{ w(x)}{ w(y)}\Big)\bigg|^p\,d\mu+c r^{N-p}.
\end{split}
\end{equation}
\textbf{Estimate of $I_3$:} Here following the proof of the estimate of $I_2$ in \cite[Pages 1817--1818]{DKPhar}, we obtain
\begin{equation}\label{estI3}
I_3\leq cr^{N-p},
\end{equation}
with $c=c(N,p,s)>0$.\\
\textbf{Estimate of $I_4$:} Since $g\in L^\infty(\Omega)$, we notice that
\begin{equation}\label{estI4}
I_4\leq cr^{N-p},
\end{equation}
with $c=c(N,g)>0$ and using that $0<r\leq 1$.\\
By using \eqref{EstI1}, \eqref{EstI2}, \eqref{estI3} and \eqref{estI4} in \eqref{DGLtsteqn1}, we obtain 
\begin{equation}\label{estI123}
\begin{split}
\int_{B_{6r}(x_0)}|\nabla\log w|^p\,dx
+\int_{B_{6r}(x_0)}\int_{B_{6r}(x_0)}\bigg|\log\Big(\frac{ w(x)}{ w(y)}\Big)\bigg|^p\,d\mu\leq cr^{N-p},
\end{split}
\end{equation}
for some constant $c=c(N,p,s,g,C_1,C_2)>0$. Then repeating the arguments after \cite[(7.11), page 21]{GK} the estimate \eqref{exp1} follows.\\
\textbf{Step 2.} Here we claim that, for every $\eta>0$, there exists a constant $\delta=\delta(N,p,s,\tau,g,C_1,C_2)\in(0,\frac{1}{4})$ such that
\begin{equation}\label{expanaux}
\inf_{B_{4r}(x_0)}\,u\geq\delta k-\Big(\frac{r}{R}\Big)^\frac{p}{p-1}\Tail(u_{-};x_0,R)-2\eta.
\end{equation}
As a corollary of \eqref{expanaux}, the property \eqref{expc} follows. Now to establish \eqref{expanaux}, without loss of generality, we assume that
\begin{equation}\label{assume}
\delta k\geq\Big(\frac{r}{R}\Big)^\frac{p}{p-1}\Tail(u_{-};x_0,R)+2\eta,
\end{equation}
since in the other case, the estimate \eqref{expanaux} holds true due to the given condition that $u\geq 0$ in $B_R(x_0)$. Let $\rho\in[r,6r]$ and $\psi\in C_c^{\infty}(B_{\rho}(x_0))$ be a cutoff function such that $0\leq\psi\leq 1$ in $B_{\rho}(x_0)$.
For any $l\in(\delta k,2\delta k)$, from Lemma \ref{wheng} and the proof of \cite[Pages 1820--1821, Lemma 3.2]{DKPhar} for $w=(l-u)_{+}$, 
for some constant $c=c(N,p,s,C_1,C_2)>0$, we obtain
\begin{equation}\label{energyapp}
\begin{split}
&\int_{B_{\rho}(x_0)}\psi^p|\nabla\,w|^p\,dx\\
&\leq c\int_{B_{\rho}(x_0)}w^p\,|\nabla\psi|^p\,dx+c\iint_{B_{\rho}(x_0)\times B_{\rho}(x_0)}\max\{w(x),w(y)\}^p|\psi(x)-\psi(y)|^p\,d\mu\\
&\qquad+c\,l\sup_{x\in\supp\psi}\int_{{\mathbb{R}^n}\setminus B_{\rho}(x_0)}\big(l+(u(y))_{-}\big)^{p-1}\,K(x,y)\,dy
\cdot|B_{\rho}(x_0)\cap\{u<l\}|\\
&\quad-\int_{B_\rho(x_0)}g(x)\,|u|^{p-2}u\,w\,\psi^p\,dx\\
&=J_1+J_2+J_3+J_4.
\end{split}
\end{equation}
Given $j=0,1,2,\dots$, we define
\begin{equation}\label{parameter}
l=k_j=\delta k+2^{-j-1}\delta k,
\quad
\rho=\rho_j=4r+2^{1-j}r,
\quad
\hat{\rho_j}=\frac{\rho_j+\rho_{j+1}}{2}.
\end{equation}
Therefore $l\in(\delta k,2\delta k)$, $\rho_j,\hat{\rho_j}\in(4r,6r)$ and 
$$
k_j-k_{j+1}=2^{-j-2}\delta k\geq 2^{-j-3}k_j
$$
for all $j=0,1,2,\dots$. We set $B_j=B_{\rho_j}(x_0),\,\hat{B}_j=B_{\hat{\rho}_j}(x_0)$ and notice that
$$
w_j=(k_j-u)_{+}\geq 2^{-j-3} k_j\chi_{\{u<k_{j+1}\}}.
$$
Let $(\psi_j)_{j=0}^{\infty}\subset C_c^{\infty}(\hat{B}_j)$ be a sequence of cutoff functions such that
$0\leq\psi_j\leq 1$ in $\hat{B}_j$, $\psi_j\equiv 1$ in $B_{j+1}$ and $|\nabla\psi_j|\leq\frac{2^{j+3}}{r}$.
We choose $\psi=\psi_j$, $w=w_j$ in \eqref{energyapp}. Then by the properties of $\psi_j$, we have
\begin{equation}\label{I1jest}
J_1=\int_{B_j}w_j^p|\nabla\psi_j|^p\,dx\leq{c2^{jp}}k_j^{p}r^{-p}|B_j\cap\{u<k_j\}|,
\end{equation}
for some constant $c=c(N,p,s,C_1,C_2)>0$. Now following the same proof of \cite[Page 1822, Lemma 3.2]{DKPhar}, for any $r\in(0,1]$, we get 
\begin{equation}\label{I2jest}
\begin{split}
J_2&=\int_{B_j}\int_{B_j}\max\{w_j(x),w_j(y)\}^p|\psi_j(x)-\psi_j(y)|^p\,d\mu\leq c\,2^{jp} k_j^{p}r^{-p}|B_j\cap\{u<k_j\}|
\end{split}
\end{equation}
for some constant $c=c(N,p,s,C_1,C_2)>0$. We estimate $J_3$ by following the proof of \cite[Page 1823, Lemma 3.2]{DKPhar}. To this end, we notice that, for every $x\in\mathrm{spt}\,\psi_j\subset\hat{B}_j$ and $y\in\mathbb{R}^n\setminus B_j$, we have
\begin{equation}\label{ne}
\frac{|y-x_0|}{|y-x|}=\frac{|y-x+x-x_0|}{|y-x|}\leq 1+\frac{|x-x_0|}{|y-x|}\leq 1+\frac{\hat{\rho}_j}{\rho_j -\hat{\rho}_j}=2^{j+4}.
\end{equation}
Using \eqref{ne} and the properties of the kernel $K(x,y)=|x-y|^{-N-ps}$, we obtain
\begin{equation}\label{ne1}
\begin{split}
&\esssup_{x\in\supp\psi_j}\int_{\mathbb{R}^n\setminus B_j}\big(k_j+(u(y))_{-}\big)^{p-1}\,K(x,y)\,dy\\
&\qquad\leq c2^{j(N+ps)}\int_{\mathbb{R}^n\setminus B_j}\big(k_j+(u(y))_{-}\big)^{p-1}|y-x_0|^{-N-ps}\,dy\\
&\qquad\leq c2^{j(N+ps)}\bigg(k_j^{p-1}r^{-ps}+\int_{\mathbb{R}^n\setminus B_R(x_0)}(u(y))_{-}^{p-1}|y-x_0|^{-N-ps}\,dy\bigg)\\
&\qquad= c2^{j(N+ps)}\Big(k_j^{p-1}r^{-p}+r^{-p}\big(\frac{r}{R}\big)^p \mathrm{Tail}(u_-;x_0,R)^{p-1}\Big)\\
&\qquad\leq c2^{j(N+ps)}k_j^{p-1}r^{-p},
\end{split}
\end{equation}
with $c=c(N,p,s,C_1,C_2)$.
Here we also used the fact that $r\in(0,1]$ along with \eqref{assume}, $\delta k<k_j$ and the fact that $u\geq 0$ in $B_R(x_0)$.
Therefore, from \eqref{ne1}, we obtain
\begin{equation}\label{I3jest}
\begin{split}
J_3&=ck_j\esssup_{x\in\supp\psi_j}\int_{\mathbb{R}^n\setminus B_j}\big(k_j+(u(y))_{-}\big)^{p-1}\,K(x,y)\,dy
\cdot|B_j\cap\{u<k_j\}|\\
&\leq c2^{j(N+ps)}k_j^{p}r^{-p}|B_j\cap\{u<k_j\}|,
\end{split}
\end{equation}
with $c=c(N,p,s,C_1,C_2)$. Now, on the other hand, recalling that $r\leq 1$ along with that $a\in L^\infty(\Omega)$, we have
\begin{equation}\label{J3est}
J_4\leq c\,k_j^{p}r^{-p}|B_j\cap\{u<k_j\}|,
\end{equation}
for some constant $c=c(N,p,s,g,C_1,C_2)>0$.
By using \eqref{I1jest}, \eqref{I2jest}, \eqref{I3jest} and \eqref{J3est} in \eqref{energyapp} and taking into account \eqref{exp1}, one can follow the arguments of the proof of \cite[Lemma 7.1]{GK} to obtain the desired estimate \eqref{expanaux}. Then the final estimate \eqref{expc} follows from \eqref{expanaux}. 
\end{proof}

By Theorem \ref{whthm} and using a classical covering argument, we obtain
\begin{Corollary}\label{whthmcor}
Let $g\in L^\infty(\Omega)$ be nonpositive in $\Omega$. Let $u\in W^{1,p}_{\mathrm{loc}}(\Omega)\cap L^{ps}_{p-1}(\mathbb{R}^N)$ be a weak supersolution of \eqref{wheqn} in $\Omega$ such that $u\geq 0$ in $\mathbb{R}^N\setminus\Omega$ and
$$
u>0\text{ on every open ball }B\Subset\Omega.
$$
Then for every open set $\mathcal{O}\Subset\Omega$, there exists a constant $C=C(\mathcal{O},u)>0$ such that 
$$
u\geq C(\mathcal{O},u)>0\text{ in }\mathcal{O}.
$$
\end{Corollary}

\subsubsection{Semicontinuity}
In this subsection, we establish semicontinuity results of weak solutions of \eqref{wheqn}. More precisely, we show that weak supersolutions of \eqref{wheqn} has a lower semicontinuous representative and weak subsolutions of \eqref{wheqn} has an upper semicontinuous representative in $\Omega$. Before we state these results below, we discuss a useful result from Liao \cite{Liao}.\\
Let $u$ be a measurable function that is locally essentially bounded below in $\Omega$. 
Let $\rho\in(0,1]$ be such that $B_\rho(y)\Subset\Omega$. Assume that $b,c\in(0,1)$, $M>0$ and
$\mu_-\leq\inf_{B_\rho(y)}u$.
Following \cite{Liao}, we say that $u$ satisfies the property $(\mathcal{D})$, if there exists a constant $\tau\in(0,1)$ depending on $b,M,\mu_-$ and other data 
(may depend on the partial differential equation and will be made precise in Lemma \ref{DGL}), but independent of $\rho$, such that
$$
|\{u\leq\mu_-+M\}\cap B_\rho(y)|\leq\tau|B_\rho(y)|,
$$
implies that $u\geq\mu_-+b\,M$ in $B_{c\rho}(y)$.

Moreover, for $u\in L^1_{\loc}(\Omega)$, we denote the set of Lebesgue points of $u$ by
$$
\mathcal{E}=\bigg\{x\in\Omega:|u(x)|<\infty,\,\lim_{r\to 0}\fint_{B_r(x)}|u(x)-u(y)|\,dy=0\bigg\}.
$$
Note that, by the Lebesgue differentiation theorem, $|\mathcal{E}|=|\Omega|$.

The following result follows from \cite[Theorem 2.1]{Liao}.
\begin{Lemma}\label{lscthmLiao}
Let $u$ be a measurable function that is locally integrable and locally essentially bounded below in $\Omega$. Assume that $u$ satisfies the property $(\mathcal{D})$. 
Then $u(x)=u_*(x)$ for every $x\in\mathcal{E}$, where
\[
u_*(x)=\lim_{r\to 0}\inf_{y\in B_r(x)}u(y).
\]
In particular, $u_*$ is a lower semicontinuous representative of $u$ in $\Omega$.
\end{Lemma}

Since $u$ is assumed to be locally essentially bounded below, the  lower semicontinuous regularization $u_*(x)$ is well defined at every point $x\in\Omega$. Now we state the main regularity results of this subsection, which are consequences of Lemma \ref{lscthmLiao} and Lemma \ref{DGL} below.

\begin{Theorem}\label{lscthm1n}(\textbf{Lower semicontinuity}).
Let $g\in L^\infty(\Omega)$ and $u\in W^{1,p}_{\mathrm{loc}}(\Omega)\cap L^{ps}_{p-1}(\mathbb{R}^N)$ which is bounded below in $\mathbb{R}^N$ be a weak supersolution of \eqref{wheqn}. 
Then 
\[
u(x)=u_*(x)=\lim_{r\to 0}\inf_{y\in B_r(x)}u(y)
\]
for every $x\in\mathcal{E}$. 
In particular, $u_*$ is a lower semicontinuous representative of $u$ in $\Omega$. 
\end{Theorem}

We observe that if $u$ is a weak subsolution of \eqref{wheqn}, then $-u$ is a weak supersolution of \eqref{wheqn}. Therefore, as a Corollary of Theorem \ref{lscthm1n}, we have the following result.
\begin{Corollary}\label{uscthm}(\textbf{Upper semicontinuity}).
Let $g\in L^\infty(\Omega)$ and $u\in W^{1,p}_{\mathrm{loc}}(\Omega)\cap L^{ps}_{p-1}(\mathbb{R}^N)$ which is bounded above in $\mathbb{R}^N$ be a weak subsolution of  \eqref{wheqn}. 
Then 
\[
u(x)=u^*(x)=\lim_{r\to 0}\sup_{y\in B_r(x)}u(y)
\]
for every $x\in\mathcal{E}$. 
In particular, $u^*$ is an upper semicontinuous representative of $u$ in $\Omega$.
\end{Corollary}

We prove a De Giorgi type lemma for weak supersolutions of \eqref{wheqn}.

\begin{Lemma}\label{DGL}(\textbf{De Giorgi type lemma: I})
Let $g\in L^\infty(\Omega)$ and $u\in W^{1,p}_{\mathrm{loc}}(\Omega)\cap L^{ps}_{p-1}(\mathbb{R}^N)$ which is bounded below in $\mathbb{R}^N$ be a weak supersolution of \eqref{wheqn}. 
Let $M>0$, $b\in(0,1)$, $B_r(x_0)\Subset\Omega$ with $r\in(0,1]$ and
$\mu_-\leq\inf_{B_r(x_0)}\,u,\,\lambda_-\leq\inf_{\R^N}\,u$.
There exists a constant $\tau=\tau(N,p,s,g,b,M,\mu_-,\lambda_-,C_1,C_2)\in(0,1)$ such that if
$$
|\{u\leq\mu_-+M\}\cap B_r(x_0)|\leq\tau|B_r(x_0)|,
$$
then $u\geq\mu_-+b\,M$ in $B_{\frac{3r}{4}(x_0)}$.
\end{Lemma}

\begin{proof}
We closely follow the proof of \cite[Lemma 9.4]{GK}. To this end, without loss of generality, we assume that $x_0=0$. For $j=0,1,2,\dots$, we denote
\begin{equation}\label{parameter1}
\begin{split}
k_j&=\mu_{-}+b\,M+\frac{(1-b)M}{2^j},
\quad
\bar{k}_j=\frac{k_j+k_{j+1}}{2},\\
\quad
r_j&=\frac{3r}{4}+\frac{r}{2^{j+2}},
\quad\bar{r}_j=\frac{r_j+r_{j+1}}{2},
\end{split}
\end{equation}
$B_j=B_{r_j}(0)$, $\bar{B}_j=B_{\bar{r}_j}(0)$, $w_j=(k_j-u)_{+}$ and $\bar{w}_j=(\bar{k}_j-u)_{+}$.
We observe that $B_{j+1}\subset\bar{B}_j\subset B_j$, $\bar{k}_j<k_j$ and hence $\bar{w}_j\leq w_j$. 
Let $(\psi_{j})_{j=0}^{\infty}\subset C_{c}^{\infty}(\bar{B}_j)$ be a sequence of cutoff functions satisfying $0 \leq \psi_{j} \leq 1$ in $B_j$, $\psi_{j}= 1$ in $B_{j+1}$,
$|\nabla\psi_j|\leq\frac{2^{j+3}}{r}$. By applying Lemma \ref{wheng} to $w_j$, we obtain
\begin{equation}\label{energyapp2}
\int_{B_j}|\nabla w_j|^p \psi_j^p\,dx\leq
C(I_1+I_2+I_3+I_4),
\end{equation}
for some constant $C=C(N,p,s,C_1,C_2)>0$, where
\[
I_1=\int_{B_{j}}\int_{B_{j}}{\max\{w_j(x),w_j(y)\}^p|\psi_{j}(x)-\psi_{j}(y)|^p}\,d\mu,
\quad I_2=\int_{B_j}w_j^p|\nabla\psi_j|^p\,dx,
\]

\[
I_3=\sup_{x\in\supp\psi_j}\int_{{\mathbb{R}^n\setminus B_{j}}}{\frac{w_j(y)^{p-1}}{|x-y|^{n+ps}}}\,dy
\cdot\int_{B_{j}}w_j\psi_j^p\,dx,
\]
and
\[
I_4=-\int_{B_j}g(x)\,|u|^{p-2}u\,w_j\,\psi_j^{p}\,dx.
\]
Since $u\geq \lambda_-$ in $\R^n$, noting the definition of $k_j$ from above, we have
\begin{equation}\label{nelsc}
w_j=(k_j-u)_+\leq (M+\mu_--\lambda_-)_+=L\text{ in }\R^n. 
\end{equation}
Let $A_j=\{u<k_j\}\cap B_j$. Then from the estimates of $I_1, I_2$ and $I_3$ in the proof of \cite[Lemma 9.4]{GK}, it follows that
\begin{equation}\label{whestI1}
\begin{split}
I_1+I_2+I_3&\leq C\frac{2^{j(N+p)}}{r^{p}}L^p|A_j|,
\end{split}
\end{equation}
for every $r\in(0,1]$ for some constant $C=C(N,p,s,C_1,C_2)>0$. Now, we observe that
\begin{equation}\label{whestI4}
I_4\leq C\frac{S^{p-1}L}{r^{p}}|A_j|,
\end{equation}
where $S=\max\{|\mu_-|,|\mu_-+M|\}$ and $L=(M+\mu_--\lambda_-)_+$ for some positive constant $C=C(N,p,s,g,C_1,C_2)$. Plugging the estimates \eqref{whestI1} and \eqref{whestI4} in \eqref{energyapp2}, we have
\begin{equation}\label{energyapp4}
\begin{split}
\int_{B_j}|\nabla w_j|^p \psi_j^p\,dx\leq C\frac{2^{j(N+p)}}{r^{p}}L^p\Big(1+\big(\frac{S}{L}\big)^{p-1}\Big)|A_j|,
\end{split}
\end{equation}
for some constant $C=C(N,p,s,g,C_1,C_2)>0$. Now proceeding similarly as in the proof of \cite[Lemma 9.4]{GK}, we obtain the following iterative scheme
$$
Y_{j+1}\leq \frac{CL\Big(2+\big(\frac{S}{L})^{p-1}\Big)}{(1-b)M}\,2^{j\big(2+\frac{N+p}{p}\big)}\,Y_j^{1+\frac{1}{p}(1-\frac{1}{\kappa})},
$$
where $C=C(N,p,s,g,C_1,C_2)>0$ and $\kappa$ is given by \eqref{kappa}. Here $Y_j=\frac{|A_j|}{|B_j|}$ for $j=0,1,2,\ldots$. By choosing
\begin{align*}
c_0&=\frac{CL\Big(2+\big(\frac{S}{L})^{p-1}\Big)}{(1-b)M},
\quad
b=2^{(2+\frac{N+p}{p})},
\quad
\beta=\frac{1}{p}\Big(1-\frac{1}{\kappa}\Big),
\quad
\tau=c_0^{-\frac{1}{\beta}}b^{-\frac{1}{\beta^2}}\in(0,1)
\end{align*}
in Lemma \ref{iteration} gives $Y_j\to0$ as $j\to\infty$, if $Y_0\leq\tau$.
This implies that  $u\geq\mu_{-}+b\,M$ almost everywhere in $B_{\frac{3r}{4}}(0)$.
\end{proof}

\subsection{Doubly nonlinear parablic case}\label{dls}
\subsubsection{{\color{blue}Caccioppoli} type energy estimates}
For $T>0$, we study the following doubly nonlinear mixed equation
\begin{equation}\label{dleqn}
\partial_t (|u|^{\alpha-1}u)-H_p u+\epsilon\,(-\Delta_p)^s u=g(x,t)|u|^{p-2}u\text{ in }\Omega\times (0,T), 
\end{equation}
where $\epsilon\in(0,1]$, $\alpha>$ and $g\in L^\infty_{\mathrm{loc}}(\Omega\times (0,T))$.

\begin{Definition}\label{dlwksol}(Weak solution)
Let $g\in L^\infty_{\mathrm{loc}}(\Omega\times (0,T))$. We say that $u\in L^\infty_{\mathrm{loc}}((0,T); L^{ps}_{p-1}(\mathbb{R}^N))\cap C_{\mathrm{loc}}((0,T);L^{\alpha+1}_{\mathrm{loc}}(\Omega))\cap L^p_{\mathrm{loc}}((0,T);W^{1,p}_{\mathrm{loc}}(\Omega))$ is a weak sub (or super) solution of \eqref{dleqn} if for every $\Omega'\Subset\Omega$ and every $[t_1,t_2]\subset(0,T)$, we have 
\begin{equation}\label{dlwksoleqn}
    \begin{split}
        &-\int_{t_1}^{t_2}\int_{\Omega'}|u|^{\alpha-1}u\, \partial_t\phi\,dx dt+\int_{\Omega'}|u|^{\alpha-1}u(x,t_2)\phi(x,t_2)\,dx\\
        &\quad-\int_{\Omega'}|u|^{\alpha-1}u(x,t_1)\phi(x,t_1)\,dx+\int_{t_1}^{t_2}\int_{\Omega'}H(\nabla u)^{p-1}\nabla H(\nabla u)\nabla\phi\,dx dt\\
        &\quad +\epsilon\int_{t_1}^{t_2}\iint_{\mathbb{R}^{2N}}|u(x,t)-u(y,t)|^{p-2}(u(x,t)-u(y,t))(\phi(x,t)-\phi(y,t))\,d\mu dt\\
        &-\int_{t_1}^{t_2}\int_{\Omega'}g(x)|u|^{p-2}u\,\phi\,dx dt\leq (\text{ or }\geq)\,0,
    \end{split}
\end{equation}
for every nonnegative $\phi\in L^p_{\mathrm{loc}}((0,T);W_0^{1,p}(\Omega'))\cap W^{1,\alpha+1}_{\mathrm{loc}}((0,T);L^{\alpha+1}(\Omega'))$. We say that $u\in L^\infty_{\mathrm{loc}}((0,T); L_{ps}^{p-1}(\mathbb{R}^N))\cap C((0,T);L^{\alpha+1}_{\mathrm{loc}}(\Omega))\cap L^p_{\mathrm{loc}}((0,T);W^{1,p}_{\mathrm{loc}}(\Omega))$ is a weak solution of \eqref{dleqn} if the L.H.S. integral in \eqref{dlwksoleqn} is equal to zero for every $\phi\in L^p_{\mathrm{loc}}((0,T);W_0^{1,p}(\Omega'))\cap W^{1,\alpha+1}_{\mathrm{loc}}((0,T);L^{\alpha+1}(\Omega'))$ without any sign restriction.
\end{Definition}
We remark that Definition \ref{dlwksol} is well stated by Lemma \ref{l1} and Lemma \ref{locnon1}. In the purely local setting, various regularity results have been established for the Trudinger's equation
$$
\partial_t (|u|^{p-2}u)-\Delta_p u=0
$$
in \cite{LiaoJFA, Vespri, KKma} and further extended in \cite{Bogelein, Bog1} and the references therein. The nonlocal Trudinger's equation
$$
\partial_t (|u|^{p-2}u)+(-\Delta_p)^s u=0
$$
is recently studied in \cite{BGK, Harsh} and further extended in \cite{BGKccm} and the references therein to study semicontinuity results. In the mixed context, several regularity results have been obtained for the equation \eqref{dleqn} when $\alpha=1$ and $H(x)=|x|$ in \cite{Adi, GKjde, Zhangdcds, Zhangjga} and the references therein. For the parabolic mixed anisotropic setting, see \cite{GKansp} and the references therein. The doubly nonlinear mixed equation \eqref{dleqn} is recently studied. Nakamura \cite{Kenta1, Kenta2} proved local boundedness and Harnack inequality for \eqref{dleqn} assuming that $\alpha=p-1$ and $g\equiv 0$. The case $\alpha>0$ is recently studied in \cite{ZhangRad} to obtain Harnack inequality for the equation \eqref{dleqn} assuming that $H(x)=|x|$ and $g\equiv 0$.  

Below, we establish energy estimates for subsolutions and supersolutions. In the mixed case, see \cite{GKjde, Kenta1, Kenta2} and the references therein. First, we prove the following energy estimate for weak supersolutions of \eqref{dleqn}.
\begin{Theorem}\label{dlsupeng}(\textbf{Caccioppoli type inequality for supersolutions})
Let $g\in L^\infty_{\mathrm{loc}}(\Omega\times (0,T))$ and  $v\in L^\infty_{\mathrm{loc}}((0,T);L^\infty(\mathbb{R}^N))\cap C_{\mathrm{loc}}((0,T);L^{\alpha+1}_{\mathrm{loc}}(\Omega))\cap L^p_{\mathrm{loc}}((0,T);W^{1,p}_{\mathrm{loc}}(\Omega))$ be  a weak supersolution of \eqref{dleqn} in $\Omega\times(0,T)$  such that $v\geq\rho>0$ in $\mathbb{R}^N\times(t_0-r^p,t_0)$ for some constant $\rho>0$, where $(t_0-r^p,t_0)\Subset (0,T)$. Then for every $1<q<\min\{\alpha+1,p\}$, there exists a positive constant $C=C(p)$ such that
\begin{equation}\label{dlsupest1}
\begin{split}
&\sup_{t\in(t_0-r^p,t_0)}\int_{B_r(x_0)\times\{t\}}
v^{\alpha-q+1}\psi^p\,dx\\
&\leq p\iint_{Q_r^{-}(x_0)}v^{\alpha-q+1}\psi^{p-1}|\partial_t\psi|\,dx dt+\frac{C}{(q-1)^\frac{1}{p-1}}\iint_{Q_r^{-}(x_0)}v^{p-q}|\nabla\psi|^p\,dx dt\\
&\quad+(\zeta(q)+1+(q-1)^{1-p})\int_{t_0-r^p}^{t_0}\int_{B_r(x_0)}\int_{B_r(x_0)}\max\{v(x,t)^{p-q},v(y,t)^{p-q}\}|\psi(x,t)-\psi(y,t)|^p\,d\mu dt\\
&\quad+C\left(\sup_{x\in\mathrm{spt}\,(\cdot,t)}\int_{\mathbb{R}^N\setminus B_r(x_0)}\frac{dy}{|x-y|^{N+ps}}\right)\iint_{Q_r^{-}(x_0)}v^{p-q}\psi^p\,dx dt-\iint_{Q_{r}^{-}(x_0)}g\,v^{p-q}\,\psi^p\,dx dt
\end{split}   
\end{equation}
and
\begin{equation}\label{dlsupest2}
\begin{split}
&\iint_{Q_r^{-}(x_0)}v^{-q}\psi^p H(\nabla v)^p\,dx dt\\
&\leq \frac{2p\,\alpha}{(q-1)(\alpha-q+1)}\iint_{Q_r^{-}(x_0)}v^{q-p+1}\psi^{p-1}|\partial_t\psi|\,dx dt
+\frac{C}{(q-1)^\frac{p}{p-1}}\iint_{Q_r^{-}(x_0)}v^{p-q}|\nabla\psi|^p\,dx dt\\
&\quad+\frac{2(\zeta(q)+1+(q-1)^{1-p})}{q-1}\int_{t_0-r^p}^{t_0}\int_{B_r(x_0)}\int_{B_r(x_0)}\max\{v(x,t)^{p-q},v(y,t)^{p-q}\}|\psi(x,t)-\psi(y,t)|^p\,d\mu dt\\
&\quad+\frac{C}{q-1}\left(\sup_{x\in\mathrm{spt}\,(\cdot,t)}\int_{\mathbb{R}^N\setminus B_r(x_0)}\frac{dy}{|x-y|^{N+ps}}\right)\iint_{Q_r^{-}(x_0)}v^{p-q}\psi^p\,dx dt-\frac{2}{q-1}\iint_{Q_{r}^{-}(x_0)}g\,v^{p-q}\,\psi^p\,dx dt
\end{split}
\end{equation}
for every nonnegative $\psi\in C_c^{\infty}(Q_{r}^{-}(x_0))$, where $Q_r^{-}(x_0)=B_r(x_0)\times (t_0-r^p,t_0)\Subset\Omega\times(0,T)$ and 
$$
\zeta(q)=
\begin{cases}
        \frac{p^p(q-1)}{p-q}\quad \text{ if }0<p-q<1,\\
        (q-1)\big(\frac{p}{p-q}\big)^p, \quad \text{ otherwise}.
\end{cases}
$$
\end{Theorem}

\begin{proof}
    Let $1<q<\min\{\alpha+1,p\}$. Then testing \eqref{dlwksoleqn} with $\phi=v^{1-q}\psi^p$, where $\psi\in C_c^{\infty}(Q_{r}^{-}(x_0))$ is nonnegative, we obtain
    \begin{equation}\label{dlwksoleqntst1}
    \begin{split}
        I_1+I_2+I_3\geq I_4,
    \end{split}
    \end{equation}
    where
    $$
    I_1=\iint_{Q_r^{-}(x_0)}\partial_t(|v|^{\alpha-1}v)\,v^{1-q}\,\psi^p\,dx dt
    $$

    $$
    I_2=\iint_{Q_r^{-}(x_0)}H(\nabla v)^{p-1}\nabla H(\nabla v)\nabla (v^{1-q}\psi^p)\,dx dt,
    $$

    $$
    I_3=\epsilon\int_{t_0-r^p}^{t_0}\iint_{\mathbb{R}^{2N}}|v(x,t)-v(y,t)|^{p-2}(v(x,t)-v(y,t))(v(x,t)^{1-q}\psi(x,t)^p-v(y,t)^{1-q}\psi(y,t)^p)\,d\mu dt,
    $$
\quad\text{and}
    $$
    I_4=\iint_{Q_r^{-}(x_0)}g\,v^{p-q}\psi^p\,dx dt.
    $$
    We remark that the above test function $\phi$ is admissible although it depends on $v$, since one can use the exact argument using mollification from the proof of \cite[Lemma 3.1]{Kenta2}.\\
\textbf{Estimate of $I_2$:} Proceeding exactly as in the proof of the estimate of $I$ in Theorem \ref{engsup}, for some constant $C=C(p)>0$, we obtain
    \begin{equation}\label{dlestI2}
    \begin{split}
    I_2&\leq -\frac{q-1}{2}\int_{t_0-r^p}^{t_0}\int_{B_r(x_0)}v^{-q}\psi^p H(\nabla v)^p\,dx dt+\frac{C}{(q-1)^\frac{1}{p-1}}\int_{t_0-r^p}^{t_0}\int_{B_r(x_0)}v^{p-q}|\nabla\psi|^p\,dx dt.
    \end{split}
    \end{equation}
    \textbf{Estimate of $I_3$:} Proceeding exactly as in the proof of the estimate of $J$ in Theorem \ref{engsup}, for some constant $C=C(N,p,s)>0$, we obtain
    \begin{equation}\label{dlestI3}
    \begin{split}
    I_3&\leq\epsilon\,\Big(\zeta(q)+1+\frac{1}{(q-1)^{p-1}}\Big)\int_{t_0-r^p}^{t_0}\int_{B_r(x_0)}\int_{B_r(x_0)}\max\{v(x,t)^{p-q},v(y,t)^{p-q}\}|\psi(x,t)-\psi(y,t)|^p\,d\mu\\
&\quad +C\,\epsilon\left(\sup_{x\in\mathrm{spt}\,\psi(\cdot,t)}\int_{\mathbb{R}^N\setminus B_r(x_0)}\frac{dy}{|x-y|^{N+ps}}\right)\int_{t_0-r^p}^{t_0}\int_{B_r(x_0)}v^{p-q}\psi^p\,dx,
    \end{split}
    \end{equation}
    where
    $$
    \zeta(q)=
    \begin{cases}
        \frac{p^p(q-1)}{p-q}\quad \text{ if }0<p-q<1,\\
        (q-1)\big(\frac{p}{p-q}\big)^p, \quad \text{ otherwise}.
    \end{cases}
    $$
    Finally, the integral in $I_1$ can be estimated exactly as in the proof of \cite[Lemma 3.1]{Kenta2} and thus combining it with the estimates \eqref{dlestI2} and \eqref{dlestI3} in \eqref{dlwksoleqntst1} the estimates \eqref{dlsupest1} and \eqref{dlsupest2} follows. 
\end{proof}
Now, we have the following logarithmic estimate for weak supersolutions, which follows similarly as in the proof of Theorem \ref{dlsupeng} above and \cite[Lemma 3.4]{Kenta2} by choosing the test function $\phi=v^{1-p}\psi^p$.
\begin{Theorem}\label{dllogthm}(\textbf{Logarithmic Caccioppoli type inequality for supersolutions})
Let $g\in L^\infty_{\mathrm{loc}}(\Omega\times(0,T))$ and $\alpha=p-1$, $0<t_1<t_2<T$. Assume that $v\in L^\infty_{\mathrm{loc}}((0,T);L^\infty(\mathbb{R}^N))\cap C_{\mathrm{loc}}((0,T);L^{\alpha+1}_{\mathrm{loc}}(\Omega))\cap L^p_{\mathrm{loc}}((0,T);W^{1,p}_{\mathrm{loc}}(\Omega))$ is a weak supersolution of \eqref{dleqn} in $\Omega\times(0,T)$ such that there exists a constant $\rho$ with $v\geq \rho >0$ in $\mathbb{R}^N\times(t_1,t_2)$. Then there exists a positive constant $C=C(p,C_1,C_2)$ such that for every nonnegative $\psi\in C_c^{\infty}(\Omega_{t_1,t_2})$, the following quantitative estimate holds
\begin{equation}\label{dllogextine}
\begin{split}
    &\iint_{\Omega_{t_1,t_2}}\psi^p |\nabla \mathrm{log}\,v|^p\,dx dt-\Big[\int_{\Omega}\psi^p\,\mathrm{log}\,v\,dx\Big]_{t=t_1}^{t_2}\\
    &\leq C\iint_{\Omega_{t_1,t_2}}|\nabla\psi|^p\,dx dt+C\iint_{\Omega_{t_1,t_2}}\psi^{p-1}|\partial_t\psi||\mathrm{log}\,v|\,dx dt-\iint_{\Omega_{t_1,t_2}}g\,\psi^p\,dx dt\\
    &+C\,\epsilon\int_{t_1}^{t_2}\iint_{\mathbb{R}^{2N}}|u(x,t)-u(y,t)|^{p-2}(u(x,t)-u(y,t))\big(u(x,t)^{1-p}\psi(x,t)^{1-p}-u(y,t)^{1-p}\psi(y,t)^p\big)\,d\mu dt,
\end{split}    
\end{equation}
where $\Omega_{t_1,t_2}=\Omega\times(t_1,t_2)$.
\end{Theorem}
Next, we have the following estimate for weak subsolutions, whose proof follows using Picone identity similarly as the proof of Theorem \ref{engsub} and \cite[Lemma 4.1]{Kenta2}.
\begin{Theorem}\label{dlsubthm}(\textbf{Caccioppoli type inequality for subsolutions})
Let $g\in L^\infty_{\mathrm{loc}}(\Omega\times(0,T)$ and suppose $u\in L^\infty_{\mathrm{loc}}((0,T); L^\infty(\mathbb{R}^N))\cap C_{\mathrm{loc}}((0,T);L^{\alpha+1}_{\mathrm{loc}}(\Omega))\cap L^p_{\mathrm{loc}}((0,T);W^{1,p}_{\mathrm{loc}}(\Omega))$ is a weak subsolution of \eqref{dleqn} in $\Omega\times(0,T)$ such that $v>0$ in $B_r(x_0)\times(t_0,t_0+r^p)\Subset \Omega\times(0,T)$.
Then for every $q>p$, there exists a positive constant $C=C(p)$ such that
\begin{equation}\label{engsupest1}
\begin{split}
&\sup_{t\in(t_0,t_0+r^p)}\int_{B_r(x_0)\times\{t\}}
u^{\alpha+q-p+1}\eta^p\,dx\\
&\leq p \iint_{Q_r^{+}(x_0)}u^{\alpha+q-p+1}\psi^{p-1}|\partial_t\psi|\,dx dt+\frac{\alpha+q-p+1}{\alpha}\Big(\frac{p}{q-p+1}\Big)^p\iint_{Q_r^{+}(x_0)}u^q H(\nabla\eta)^p\,dx dt\\
&+\frac{\alpha+q-p+1}{\alpha}\frac{p}{q-p+1}\iint_{Q_r^{+}(x_0)}g\,u^q\,\eta^p\,dx dt\\
&\quad+\frac{C\,p}{q-p+1}\frac{\alpha+q-p+1}{\alpha}\int_{t_0}^{t_0+r^p}\int_{B_r(x_0)}\int_{B_r(x_0)}(u(x,t)^q+u(y,t)^q)|\eta(x,t)-\eta(y,t)|^p\,d\mu dt\\
&\quad+\frac{2p}{q-p+1}\frac{\alpha+q-p+1}{\alpha}\left(\sup_{x\in\mathrm{spt}\,\eta(\cdot,t)}\int_{\mathbb{R}^N\setminus B_r(x_0)}\frac{|u(y,t)|^{p-1}}{|x-y|^{N+ps}}\,dy\right)\iint_{Q_r^{+}(x_0)}u^{q-p+1}\,\eta^p\,dx dt,
\end{split}
\end{equation}
and
\begin{equation}\label{engsupest2}
\begin{split}
&\iint_{Q_r^{+}(x_0)}u^{q-p}H(\eta \nabla u)^p\,dx\\
&\leq\frac{p\,\alpha}{\alpha+q-p+1}\iint_{Q_r^{+}(x_0)}u^{\alpha+q-p+1}\psi^{p-1}|\partial_t\psi|\,dx dt+\Big(\frac{p}{q-p+1}\Big)^p\iint_{Q_r^{+}(x_0)}u^q\,H(\nabla\eta)^p\,dx dt\\
&\quad+\frac{p}{q-p+1}\iint_{Q_r^{+}(x_0)}g\,u^q\,\eta^p\,dx dt\\
&\quad+\frac{C\,p}{q-p+1}\int_{t_0}^{t_0+r^p}\int_{B_r(x_0)}\int_{B_r(x_0)}(u(x,t)^q+u(y,t)^q)|\eta(x)-\eta(y)|^p\,d\mu dt\\
&\quad+\frac{2p}{q-p+1}\left(\sup_{x\in\mathrm{spt}\,\eta(\cdot,t)}\int_{\mathbb{R}^N\setminus B_r(x_0)}\frac{|u(y,t)|^{p-1}}{|x-y|^{N+ps}}\,dy\right)\iint_{Q_r^{+}(x_0)}u^{q-p+1}\,\eta^p\,dx dt,
\end{split}
\end{equation}
for every nonnegative $\eta\in C_c^{\infty}(Q_r^{+}(x_0))$, where $Q_r^{+}(x_0)=B_r(x_0)\times (t_0,t_0+r^p)$.
\end{Theorem}

\subsubsection{Semicontinuity and pointwise behavior}
In this subsection, we obtain lower semicontinuity as well as upper semicontinuity results for weak supersolution and subsolution of \eqref{dleqn} respectively. We also discuss the pointwise behavior of such semicontinuous representatives. First, we define the lower and upper semicontinuous representatives of a measurable function. In the local case, see \cite{Liao}, in the nonlocal case, see \cite{BGKccm} and in the mixed setting, we refer to \cite{GKansp} and the references therein.

Let $u$ be a measurable function that is locally essentially bounded below in $\Omega_T=\Omega\times(0,T)$. Suppose that $(x,t)\in\Om_T$ and $r\in(0,1)$, $\theta>0$ such that $\mathcal{Q}_{r,\theta}(x,t)=B_r(x)\times(t-\theta r^p,t+\theta r^p)\Subset\Om_T$. The  lower semicontinuous regularization $u_*$ of $u$ is defined as 
\begin{equation}\label{lscrp}
u_*(x,t)=\essliminf_{(y,\hat{t})\to (x,t)}\,u(y,\hat{t})=\lim_{r\to 0}\inf_{\mathcal{Q}_{r,\theta}(x,t)}\,u\text{ for } (x,t)\in \Om_T.
\end{equation}

Analogously, for a locally essentially bounded above measurable function $u$ in $\Om_T$, we define an upper semicontinuous regularization $u^*$  of $u$ by
\begin{equation}\label{uscrp}
u^*(x,t)=\esslimsup_{(y,\hat{t})\to (x,t)}\,u(y,\hat{t})=\lim_{r\to 0}\sup_{\mathcal{Q}_{r,\theta}(x,t)}\,u\text{ for } (x,t)\in \Om_T.
\end{equation}
It is easy to see that $u_*$ is lower semicontinuous and  $u^*$ is upper semicontinuous in $\Om_T$.

Let $u\in L^1_{\mathrm{loc}}(\Om_T)$ and define the set of Lebesgue points of $u$ by
$$
\mathcal{F}=\bigg\{(x,t)\in\Om_T:|u(x,t)|<\infty,\,\lim_{r\to 0}\fint_{\mathcal{Q}_{r,\theta}(x,t)}|u(x,t)-u(y,\hat{t})|\,dy\,d\hat{t}=0\bigg\}.
$$
From the Lebesgue differentiation theorem we have $|\mathcal{F}|=|\Om_T|$. \\

.
The main results of this subsection reads as follows:
\begin{Theorem}\label{lscthm1}(\textbf{Lower semicontinuity})
Let $g\in L^\infty(\Omega_T)$ and $u\in L^p_{\mathrm{loc}}\big(0,T;W^{1,p}_{\mathrm{loc}}(\Om)\big)\cap C_{\mathrm{loc}}\big(0,T;L^{\alpha+1}_{\mathrm{loc}}(\Om)\big)\cap L^\infty_{\mathrm{loc}}\big(0,T;L^{p-1}_{ps}(\mathbb{R}^N)\big)$ be a weak supersolution of \eqref{dleqn} in $\Om_T$ such that $u$ is essentially bounded below in $\R^N\times(0,T)$. Let $u_*$ be defined by \eqref{lscrp}. Then $u(x,t)=u_*(x,t)$ at every Lebesgue point $(x,t)\in\Om_T$. 
In particular, $u_*$ is a lower semicontinuous representative of $u$ in $\Om_T$.
\end{Theorem}

\begin{proof}
By Lemma \ref{DGL1}, the function $u$ satisfies the property $(\mathcal{P})$. Then applying Theorem \ref{lscthm}, the result follows. 
\end{proof}

\begin{Theorem}\label{lscpt}(\textbf{Pointwise behavior})
Let $g\in L^\infty(\Om_T)$ and $u\in L^p_{\mathrm{loc}}\big(0,T;W^{1,p}_{\mathrm{loc}}(\Om)\big)\cap C_{\mathrm{loc}}\big(0,T;L^{\alpha+1}_{\mathrm{loc}}(\Om)\big)\cap L^\infty_{\mathrm{loc}}\big(0,T;L^{p-1}_{ps}(\mathbb{R}^N)\big)$ be a weak supersolution of \eqref{dleqn} in $\Om_T$ such that $u$ is essentially bounded below in $\R^N\times(0,T)$.
Assume that $u_*$ is the lower semicontinuous representative of $u$ given by Theorem \ref{lscthm1}. 
Then for every $(x,t)\in\Om_T$, we have
$$
u_*(x,t)=\inf_{\theta>0}\lim_{r\to 0}\essinf_{\mathcal{Q}'_{r,\theta}(x,t)}\,u,
$$
where $\mathcal{Q}'_{r,\theta}(x,t)=B_r(x)\times(t-2\theta r^{p},t-\theta r^{p}),\,r\in(0,1)$.
In particular, we have
$$
u_*(x,t)={\essliminf_{(y,\hat{t})\to(x,t),\,\hat{t}<t}}\,u(y,\hat{t})
$$
at every point $(x,t)\in\Om_T$.
\end{Theorem}

\begin{proof}
Taking into account Lemma \ref{DGL2} below and proceeding along the lines of the proof of \cite[Theorem 3.1]{Liao}, the result follows.
\end{proof}


Taking into account that $-u$ is a weak supersolution of \eqref{dleqn} if $u$ is a weak subsolution of \eqref{dleqn}, the following two results follows as a consequence of Theorem \ref{lscthm1} and Theorem \ref{lscpt}.
\begin{Theorem}\label{uscthm1}(\textbf{Upper semicontinuity})
Let $g\in L^\infty(\Om_T)$ and $u\in L^p_{\mathrm{loc}}\big(0,T;W^{1,p}_{\mathrm{loc}}(\Om)\big)\cap C_{\mathrm{loc}}\big(0,T;L^{\alpha+1}_{\mathrm{loc}}(\Om)\big)\cap L^\infty_{\mathrm{loc}}\big(0,T;L^{p-1}_{ps}(\mathbb{R}^N)\big)$ is a weak subsolution of \eqref{dleqn} in $\Om_T$ such that $u$ is essentially bounded above in $\R^N\times(0,T)$. Let $u^*$ be defined by \eqref{uscrp}. Then $u(x,t)=u^*(x,t)$ at every Lebesgue point $(x,t)\in\Om_T$. In particular, $u^*$ is an upper semicontinuous representative of $u$ in $\Om_T$.
\end{Theorem}



Our final result concerns the pointwise behavior of the upper semicontinuous representative given by Theorem \ref{uscthm1}.

\begin{Theorem}\label{uscpt}(\textbf{Pointwise behavior})
Let $g\in L^\infty(\Om_T)$ and $u\in L^p_{\mathrm{loc}}\big(0,T;W^{1,p}_{\mathrm{loc}}(\Om)\big)\cap C_{\mathrm{loc}}\big(0,T;L^{\alpha+1}_{\mathrm{loc}}(\Om)\big)\cap L^\infty_{\mathrm{loc}}\big(0,T;L^{p-1}_{ps}(\mathbb{R}^N)\big)$ be a weak subsolution of \eqref{dleqn} in $\Om_T$ such that $u$ is essentially bounded above in $\R^N\times(0,T)$.
Assume that $u^*$ is the upper semicontinuous representative of $u$ given by Theorem \ref{uscthm1}. 
Then for every $(x,t)\in\Om_T$, we have
$$
u^*(x,t)=\sup_{\theta>0}\lim_{r\to 0}\sup_{\mathcal{Q}'_{r,\theta}(x,t)}\,u,
$$
where $\mathcal{Q}'_{r,\theta}(x,t)=B_r(x)\times(t-2\theta r^{p},t-\theta r^{p}),\,r\in(0,1)$.
In particular, we have
$$
u^*(x,t)={\esslimsup_{(y,\hat{t})\to(x,t),\,\hat{t}<t}}\,u(y,\hat{t})
$$
at every point $(x,t)\in\Om_T$.
\end{Theorem}

\textbf{Preliminaries:} The following measure theoretic property from \cite{Liao} will be useful for us.

\begin{Definition}\label{propertyP} Let $u$ be a measurable function which is locally essentially bounded below in $\Om_T$. Assume that $(x_0,t_0)\in\Om_T$ and $r\in(0,1)$, $\theta>0$ such that $\mathcal{Q}_{r,\theta}(x_0,t_0)=B_r(x_0)\times(t_0-\theta r^p,t_0+\theta r^p)\Subset\Om_T$. Suppose 
\begin{equation}\label{lmu}
a,c\in(0,1),\quad M>0,\quad \mu^-\leq\essinf_{\mathcal{Q}_{r,\theta}(x_0,t_0)}\,u.
\end{equation}
We say that $u$ satisfies the property $(\mathcal P)$ if there exists a constant $\nu\in(0,1)$, 
which depends only on $a,M,\theta,\mu^-$ and other data, but independent of $r$, such that if
\begin{equation}\label{lgc}
|\{u\leq\mu^-+M\}\cap\mathcal{Q}_{r,\theta}(x_0,t_0)|\leq\nu|\mathcal{Q}_{r,\theta}(x_0,t_0)|,
\end{equation}
then 
\begin{equation}\label{lr}
u\geq\mu^-+a\,M\text{ in }\mathcal{Q}_{cr,\theta}(x_0,t_0).
\end{equation}
\end{Definition}

Next, we state a result from Liao in \cite[Theorem 2.1]{Liao} that shows that any such function with the property $(\mathcal P)$ has a lower semicontinuous representative. 

\begin{Theorem}\label{lscthm}
Let $u$ be a measurable function in $\Om_T$ which is locally integrable and locally essentially bounded below in $\Om_T$. Assume that $u$ satisfies the property $(\mathcal P)$. Let $u_*$ be defined by \eqref{lscrp}. Then $u(x,t)=u_*(x,t)$ for every $x\in \mathcal{F}$. 
In particular, $u_*$ is a lower semicontinuous representative of $u$ in $\Om_T$.
\end{Theorem}

The proof of the following energy estimate follows in the exact way as in \cite[Proposition 2.1]{Liao} and \cite[Lemma 3.3]{GKansp}.
\begin{Theorem}\label{eng1}(\textbf{Energy estimate})
Let $g\in L^\infty(\Om_T)$ and suppose that $u\in L^p_{\mathrm{loc}}\big(0,T;W^{1,p}_{\mathrm{loc}}(\Om)\big)\\\cap C_{\mathrm{loc}}\big(0,T;L^{\alpha+1}_{\mathrm{loc}}(\Om)\big)\cap L^\infty_{\mathrm{loc}}\big(0,T;L^{p-1}_{ps}(\mathbb{R}^N)\big)$ is a weak supersolution (or subsolution) of \eqref{dleqn} in $\Om_T$.  Suppose $x_0\in\Om,\,r>0$ such that $B_r=B_r(x_0)\Subset\Om$ and $0<\tau_1<\tau_2$, $\tau>0$ such that $(\tau_1-\tau,\tau_2)\Subset(0,T)$. For $k\in\mathbb{R}$, we denote by $w_\pm=(u-k)_\pm$. Then there exists a positive constant $C=C(p,C_1,C_2)>0$ such that the estimate below holds with upper sign for supersolutions and with lower sign for subsolutions respectively:
\begin{equation}\label{eng1eqn}
\begin{split}
&\sup_{\tau_1-\tau<t<\tau_2}\int_{B_r}\zeta_{\pm}(u,k)\xi^p\,dx+\int_{\tau_1-\tau}^{\tau_2}\int_{B_r}|\nabla w_\pm|^p\xi^p\,dx\,dt\\
&\leq C\bigg(\int_{\tau_1-\tau}^{\tau_2}\int_{B_r}w_{\pm}^p|\nabla\xi|^p\,dx\,dt+\int_{\tau_1-\tau}^{\tau_2}\int_{B_r}\int_{B_r}\max\{w_{\pm}(x,t),w_{\pm}(y,t)\}^p|\xi(x,t)-\xi(y,t)|^p\,d\mu\,dt\\
&\qquad+\sup_{(x,t)\in\mathrm{spt}\,\xi,\,\tau_1-\tau<t<\tau_2}\int_{\mathbb{R}^N\setminus B_r}\frac{w_{\pm}(y,t)^{p-1}}{|x-y|^{N+ps}}\,dy\int_{\tau_1-\tau}^{\tau_2}\int_{B_r}w_{\pm}\xi^p\,dx\,dt\\
&\qquad+\int_{\tau_1-\tau}^{\tau_2}\int_{B_r}\zeta_{\pm}(u,k)|\partial_t\xi^p|\,dx\,dt+\int_{B_r}\zeta_{\pm}(u(x,\tau_1-\tau),k)\xi(x,\tau_1-\tau)^p\,dx\\
&\qquad\mp\int_{\tau_1-\tau}^{\tau_2}\int_{B_r}g\,|u|^{p-2}u\,w_\pm\,\xi^p\,dx dt\bigg),
\end{split}
\end{equation}
where $\xi(x,t)=\psi(x)\eta(t)$, with $\psi\in C_c^{\infty}(B_r)$ and $\eta\in C^\infty(\tau_1-\tau,\tau_2)$ are nonnegative functions. Here $\zeta_\pm$ is given by \eqref{zeta}.
\end{Theorem}

\textbf{De Giorgi type Lemmas for weak supersolutions:} Assume that $(x_0,t_0)\in\Om_T$ and $r\in(0,1)$, $\theta>0$ such that $\mathcal{Q}_{r,\theta}(x_0,t_0)=B_r(x_0)\times(t_0-\theta r^p,t_0+\theta r^p)\Subset\Om_T$. Let $a,\mu^-$ and $M$ be defined as in \eqref{lmu}. Then the following De Giorgi type lemma shows that a weak supersolution satisfies the property $(\mathcal P)$ in Definition \ref{propertyP}. The proof follows by combining the idea of the proof of \cite[Lemma 4.9]{GKansp} and \cite[Lemma 2.2]{Liao}.
\begin{Lemma}\label{DGL1}(\textbf{De Giorgi type lemma: II})
Let $g\in L^\infty(\Om_T)$ and suppose that $u\in L^p_{\mathrm{loc}}\big(0,T;W^{1,p}_{\mathrm{loc}}(\Om)\big)\cap C_{\mathrm{loc}}\big(0,T;L^{\alpha+1}_{\mathrm{loc}}(\Om)\big)\cap L^\infty_{\mathrm{loc}}\big(0,T;L^{p-1}_{ps}(\mathbb{R}^N)\big)$ be a weak supersolution of \eqref{dleqn} in $\Om_T$ such that $u$ is essentially bounded below in $\R^N\times(0,T)$ and let
\begin{equation}\label{lmbc}
\lambda^-\leq\essinf_{\mathbb{R}^N\times(0,T)}\,u.
\end{equation} 
Then there exists a constant $\nu=\nu(a,M,\mu^-,\lambda^-,\theta,N,p,s,\alpha,g,C_1,C_2)\in(0,1)$, such that if 
$$
|\{u\leq\mu^-+M\}\cap \mathcal{Q}_{r,\theta}(x_0,t_0)|\leq\nu|\mathcal{Q}_{r,\theta}(x_0,t_0)|,
$$
then
$$
u\geq\mu^-+a\,M\text{ in } \mathcal{Q}_{\frac{3r}{4},\theta}(x_0,t_0).
$$
\end{Lemma}
\begin{proof}
Following the proof of \cite[Lemma 4.9]{GKansp}, for $j\in\mathbb N\cup\{0\}$, we define
\begin{equation}\label{itelsc}
\begin{split}
&k_j=\mu^-+aM+\frac{(1-a)M}{2^j},\quad\hat{k}_j=\frac{k_j+k_{j+1}}{2},\\
&r_j=\frac{3r}{4}+\frac{r}{2^{j+2}},\quad\hat{r}_j=\frac{r_j+r_{j+1}}{2},\\
&B_j=B_{r_j}(x_0),\quad\hat{B}_{j}=B_{\hat{r}_j}(x_0),\\
&\Gamma_j=(t_0-\theta r_j^{p},t_0+\theta r_j^{p}),\quad\hat{\Gamma}_j=(t_0-\theta\hat{r}_j^p,t_0+\theta\hat{r}_{j}^p)\\
&\mathcal{Q}_j=B_j\times\Gamma_j,\quad\mathcal{\hat{Q}}_j=\hat{B}_j\times\hat{\Gamma}_j,\quad A_j=\mathcal{Q}_j\cap\{u\leq k_j\}.
\end{split}
\end{equation}
Notice that $r_{j+1}<\hat{r}_j<r_j$, $k_{j+1}<\hat{k}_j<k_j$ for all $j\in\mathbb N\cup\{0\}$ and therefore, we have
$B_{j+1}\subset\hat{B}_j\subset B_j$ and  $\Gamma_{j+1}\subset\hat{\Gamma}_j\subset\Gamma_j$.
Let $\{\Phi_j\}_{j=0}^{\infty}\subset C_c^{\infty}(\hat{\mathcal{Q}}_j)$ be such that
\begin{equation}\label{Philsc}
0\leq\Phi_j\leq 1,\quad |\nabla\Phi_j|\leq C\frac{2^j}{r}\text{ in }\hat{\mathcal{Q}}_j
\quad\text{and}\quad \Phi_j\equiv 1\text{ in }\mathcal{Q}_{j+1},
\end{equation}
for some constant $C=C(N,p)>0$. Notice that, over the set $A_{j+1}=\mathcal{Q}_{j+1}\cap\{u\leq k_{j+1}\}$, we have $\hat{k}_j-k_{j+1}\leq \hat{k}_j-u$.  
By integrating over the set $A_{j+1}$ and using Lemma \ref{Sobo}-(b), we obtain
\begin{equation}\label{Sobolsc}
\begin{split}
(1-a)\frac{M}{2^{j+3}}|A_{j+1}|
&\leq\int_{A_{j+1}}(\hat{k}_j-k_{j+1})\,dx\,dt
\leq\int_{\mathcal{Q}_{j+1}}(\hat{k}_j-u)\,dx\,dt
\leq\int_{\mathcal{\hat{Q}}_j}(u-\hat{k}_j)_-\Phi_j\,dx\,dt\\
&\leq\bigg(\int_{\hat{\mathcal{Q}}_j}\big((u-\hat{k}_j)_-\Phi_j\big)^{p(1+\frac{2}{N})}\,dx\,dt\bigg)^\frac{N}{p(N+2)}|A_j|^{1-\frac{N}{p(N+2)}}\\
&\leq C(I+J)^\frac{N}{p(N+2)}\hat{K}^\frac{1}{N+2}|A_j|^{1-\frac{N}{p(N+2)}},
\end{split}
\end{equation}
for some constant $C=C(N,p)>0$, where 
$$
I=\int_{\hat{\mathcal{Q}}_j}|\nabla (u-\hat{k}_j)_-|^p\,dx\,dt,
\quad 
J=\int_{\hat{\mathcal{Q}}_j}(u-\hat{k}_j)_{-}^{p}|\nabla\Phi_j|^p\,dx\,dt 
\quad\text{and}\quad
\hat{K}=\esssup_{t\in\hat{\Gamma}_j}\int_{\hat{B}_j}(u-\hat{k}_{j})_-^{2}\,dx.
$$
Note that, due to the assumption \eqref{lmbc}, we know $\lambda^-\leq\essinf_{\mathbb{R}^N\times(0,T)}\,u$, which gives
\begin{equation}\label{eqn}
(u-\hat{k}_j)_-\leq(\mu^-+M-\lambda^-)_+:=L \quad\text{in}\quad\mathbb{R}^N\times(0,T).
\end{equation}
Following the proof of \cite[Lemma 4.9]{GKansp}, we get
\begin{equation}\label{Jlsc1}
\begin{split}
J&\leq C\frac{2^{jp}}{r^p}L^p|A_j|,
\end{split}
\end{equation}
for some constant $C=C(N,p)>0$.\\
\textbf{Estimate of $I$ and $\hat{K}$:} Let $\xi_j=\psi_j\eta_j$, where $\{\psi_j\}_{j=0}^{\infty}\subset C_{c}^\infty(B_j)$ and $\{\eta_j\}_{j=0}^{\infty}\subset C_{c}^\infty(\Gamma_j)$ be such that 
\begin{equation}\label{lscctof}
\begin{split}
&0\leq\psi_j\leq 1,\quad|\nabla\psi_j|\leq C\frac{2^j}{r}\text{ in }B_j,\quad\psi_j\equiv 1\text{ in }\hat{B}_j,\quad\mathrm{dist}(\mathrm{spt}\,\psi_j,\mathbb{R}^N\setminus B_j)\geq 2^{-j-1}r,\\
&0\leq\eta_j\leq 1,\quad|\partial_t\eta_j|\leq C\frac{2^{pj}}{\theta r^p}\text{ in }\Gamma_j,\quad\eta_j\equiv 1\text{ in }\hat{\Gamma}_j,
\end{split}
\end{equation}
for some constant $C=C(N,p)>0$. Note that $g\equiv 0$ in $\Om_T\times\R$ and for $\hat{k}_j<k_j$, we have $(u-k_j)_-\geq (u-\hat{k}_j)_-$. Therefore, we set $r=r_j$, $\tau_1=t_0-\theta \hat{r}_j^{p}$, $\tau_2=t_0+\theta r_j^{p}$, $\tau=\theta r_j^{p}-\theta \hat{r}_j^{p}$ in Lemma \ref{eng1} to obtain
\begin{equation}\label{lscengeqn1}
\begin{split}
&\int_{\hat{\Gamma}_j}\int_{\hat{B}_j}|\nabla (u-\hat{k}_j)_-|^p\,dx dt+\frac{1}{\gamma}\sup_{\hat{\Gamma}_j}\int_{\hat{B}_j}(|u|+|k_j|)^{\alpha-1}\,(u-\hat{k}_j)_-^{2}\,dx\leq I_1+I_2+I_3+I_4+I_5,
\end{split}
\end{equation}
where
\begin{equation*}
\begin{split}
I_1&=C\int_{\Gamma_j}\int_{B_j}\int_{B_j}{\max\{(u-k_j)_{-}(x,t),(u-k_j)_{-}(y,t)\}^p|\xi_j(x,t)-\xi_j(y,t)|^p}\,d\mu\,dt,\\
I_2&=C\int_{\Gamma_j}\int_{B_j}(u-k_j)_-^p|\nabla\xi_j|^p\,dx\,dt,\\
I_3&=C\esssup_{x\in\mathrm{spt}\,\psi_j,\,t\in\Gamma_j}\int_{{\mathbb{R}^N\setminus B_j}}{\frac{(u-k_j)_{-}(y,t)^{p-1}}{|x-y|^{N+ps}}}\,dy
\int_{\Gamma_j}\int_{B_j}(u-k_j)_{-}\xi_{j}^p\,dx\,dt\\
I_4&=C\int_{\Gamma_j}\int_{B_j}\zeta_{-}(u,k_j)(u-k_j)_-^{2}|\partial_t\xi_j^{p}|\,dx\,dt\\
\quad\text{and}\\
I_5&=C\int_{\Gamma_j}\int_{B_j}|g(x,t)|\,|u|^{p-1}(u-k_j)_-\,\xi^p\,dx dt
\end{split}
\end{equation*}
for some constant $C=C(N,p,s,C_1,C_2)>0$. Following the proof of \cite[Lemma 4.9]{GKansp}, we get
\begin{equation}\label{I1lsc}
\begin{split}
I_1+I_2+I_3 &\leq C\frac{2^{j(N+p)}}{r^{p}}L^p|A_j|,
\end{split}
\end{equation}
for some constant $C=C(N,p,s,C_1,C_2)>0$. Now, as in \cite[Lemma 2.2]{Liao}, using Lemma \ref{zeta}, we observe that
\begin{equation}\label{I4lsc}
\begin{split}
I_4&\leq C\frac{2^{jp}}{\theta r^{p}} L^{2}\max\{S^{\alpha-1},L^{\alpha-1}\}|A_j|,
\end{split}
\end{equation}
and further, we notice that 
\begin{equation}\label{I5lsc}
\begin{split}
I_5&\leq C\frac{L\,S^{p-1}}{r^{p}}|A_j|,
\end{split}
\end{equation}
for some constant $C=C(N,p,s,C_1,C_2,\alpha,g)>0$, where
$$
L=(\mu^- +M-\lambda^-)_+,\quad S=\max\{|\mu^-|,\,|\mu^-+M|\}.
$$
Inserting \eqref{I1lsc}, \eqref{I4lsc} and \eqref{I5lsc} in \eqref{lscengeqn1} and following the same arguments as in the proof of \cite[Lemma 4.9]{GKansp} and \cite[Lemma 2.2]{Liao}, we obtain denoting by $Y_j=\frac{|A_j|}{|\mathcal{Q}_j|}$ for $j=0,1,2,\ldots$ that
$$
Y_{j+1}\leq c_0\,b^{j}Y_j^{1+\delta},
$$
where
$$
\quad\delta=\frac{1}{N+2},\quad b=2^{\frac{(N+p)^2}{(N+2)p}+1}>1
$$
and
$$
c_0=\frac{C\Big(\gamma\,\max\{S^{\alpha-1},\,L^{\alpha-1}\}\frac{L^2}{\theta}+L^p+L\,S^{p-1}\Big)}{(1-a)M}\Big(\frac{\theta}{(1-a)\,\min\{S^{\alpha-1},\, M^{\alpha-1}\}}\Big)^\frac{1}{N+2},
$$
where $\gamma=\gamma(\alpha)>0$ is given in Lemma \ref{Liaoin}.
We define $\nu={c_0}^{-\frac{1}{\delta}}b^{-\frac{1}{\delta^2}}$, which depends on $a,M,\mu^-,\lambda^-,\theta,N,p,s,\alpha,g,C_1,C_2$ such that if $Y_0\leq\nu$, then by Lemma \ref{iteration}, we have $\lim_{j\to\infty}Y_j=0$. This completes the proof.
\end{proof}
Recalling that $a,\mu^-$ and $M$ are defined as in \eqref{lmu}, we prove our second De Giorgi Lemma. 
\begin{Lemma}\label{DGL2}(\textbf{De Giorgi type lemma: III})
Let $g\in L^\infty(\Omega_T)$ and suppose that $u\in L^p_{\mathrm{loc}}\big(0,T;W^{1,p}_{\mathrm{loc}}(\Om)\big)\cap C_{\mathrm{loc}}\big(0,T;L^{\alpha+1}_{\mathrm{loc}}(\Om)\big)\cap L^\infty_{\mathrm{loc}}\big(0,T;L^{p-1}_{ps}(\mathbb{R}^N)\big)$ is a weak supersolution of \eqref{dleqn} in $\Om_T$ such that $u$ is essentially bounded below in $\R^N\times(0,T)$. Assume that
\begin{equation}\label{lmbcc}
\lambda^-\leq\essinf_{\mathbb{R}^N\times(0,T)}\,u.
\end{equation}
Then there exists a constant $\theta=\theta(a,M,\mu^-,\lambda^-,N,p,s,\alpha,g,C_1,C_2)>0$ such that if $t_0$ is a Lebesgue point of $u$ and
$$
u(\cdot,t_0)\geq \mu^-+M\text{ in }B_r(x_0),
$$
then 
$$
u\geq\mu^-+a\,M\text{ in }\mathcal{Q}^{+}_{\frac{3r}{4},\theta}(x_0,t_0)=B_{\frac{3r}{4}}(x_0)\times\big(t_0,t_0+\theta\big(\tfrac{3r}{4}\big)^p\big).
$$
\end{Lemma}
\begin{proof}
We closely follow the proof of \cite[Lemma 4.10]{GKansp} and \cite[Lemma 3.1]{Liao}. To this end, for $j\in\N\cup\{0\}$, we define $k_j,\hat{k}_j,r_j,\hat{r}_j,B_j,\hat{B}_j$ as in \eqref{itelsc} and for $\theta>0$, let us set
\begin{equation}\label{itelsc2}
\begin{split}
\Gamma_j=(t_0,t_0+\theta r_j^{p}),\quad\mathcal{Q}_j=B_j\times\Gamma_j,\quad \mathcal{\hat{Q}}_j=\hat{B}_j\times\Gamma_j,\quad A_j=\mathcal{Q}_j\cap\{u\leq k_j\}.
\end{split}
\end{equation}
Therefore, for all $j\in\mathbb N\cup\{0\}$ we have
$
B_{j+1}\subset\hat{B}_j\subset B_j,\,\Gamma_{j+1}\subset\Gamma_j.
$
Let $\{\Phi_j\}_{j=0}^{\infty}\subset C_c^{\infty}(\hat{\mathcal{Q}}_j)$ be as defined in \eqref{Philsc}. Notice that, over the set $A_{j+1}=\mathcal{Q}_{j+1}\cap\{u\leq k_{j+1}\}$, we have $\hat{k}_j-k_{j+1}\leq \hat{k}_j-u$. Hence integrating over the set $A_{j+1}$ as in the proof of \eqref{Sobolsc}, we obtain
\begin{equation}\label{Sobolsc2}
\begin{split}
(1-a)\frac{M}{2^{j+3}}|A_{j+1}|&\leq C(I+J)^\frac{N}{p(N+2)}\hat{K}^\frac{1}{N+2}|A_j|^{1-\frac{N}{p(N+2)}},
\end{split}
\end{equation}
for some constant $C=C(N,p)>0$, where
$$
I=\int_{\hat{\mathcal{Q}}_j}|\nabla (u-\hat{k}_j)_-|^p\,dx\,dt,
\quad 
J=\int_{\hat{\mathcal{Q}}_j}(u-\hat{k}_j)_{-}^{p}|\nabla\Phi_j|^p\,dx\,dt
\quad\text{and}\quad
\hat{K}=\esssup_{t\in\Gamma_j}\int_{\hat{B}_j}(u-\hat{k}_{j})_{-}^2\,dx.
$$
Due to \eqref{lmbcc}, as in \eqref{eqn}, we get $(u-\hat{k}_j)_-\leq (\mu^-+M-\lambda^-)_+:=L$ in $\R^N\times(0,T)$. As in the proof of \cite[Lemma 4.10]{GKansp}, we have
\begin{equation}\label{Jlsc2}
\begin{split}
J&=\int_{\hat{\mathcal{Q}}_j}(u-\hat{k}_j)_{-}^{p}|\nabla\Phi_j|^p\,dx\,dt\leq C\frac{2^{jp}}{r^p}L^p|A_j|,
\end{split}
\end{equation}
for some constant $C=C(N,p)>0$.\\
\textbf{Estimate of $I$ and $\hat{K}$:} Let $\xi_j(x,t)=\xi_j(x)$ be a time independent smooth function with compact support in $B_j$ such that $0\leq\xi_j\leq 1$, $|\nabla\xi_j|\leq C\frac{2^j}{r}$ in $\mathcal{Q}_j$, $\mathrm{dist}(\mathrm{spt}\,\xi_j,\mathbb{R}^N\setminus B_j)\geq 2^{-j-1}r$ and $\xi_j\equiv 1$ in $\hat{B}_j$ for some constant $C=C(N,p)>0$. Therefore, $\partial_t\xi_j=0$. Also, since $\hat{k}_j<\mu^-+M$, due to the hypothesis $u(\cdot,t_0)\geq\mu^-+M$ in $B_r(x_0)$, we deduce that $(u-\hat{k}_j)_-(\cdot,t_0)=0$ in $B_r(x_0)$. Noting these facts along with $g\equiv 0$ in $\Om_T\times\R$ and $(u-k_j)_-\geq (u-\hat{k}_j)_-$, by Lemma \ref{zeta} and Lemma \ref{eng1}, we obtain
\begin{equation}\label{KIlsc2}
\begin{split}
\frac{1}{\gamma}\sup_{\Gamma_j}\int_{\hat{B}_j}(|u|+|k_j|)^{\alpha-1}(u-\hat{k}_j)_-^{2}\,dx+\int_{\Gamma_j}\int_{\hat{B}_j}|\nabla (u-\hat{k}_j)_-|^p\,dx\,dt\leq J_1+J_2+J_3+J_4,
\end{split}
\end{equation}
where
\begin{equation*}
\begin{split}
J_1&=C\Bigg(\int_{\Gamma_j}\int_{B_j}\int_{B_j}{\max\{(u-k_j)_{-}(x,t),(u-k_j)_{-}(y,t)\}^p|\xi_j(x,t)-\xi_j(y,t)|^p}\,d\mu\,dt,\\
J_2&=C\int_{\Gamma_j}\int_{B_j}(u-k_j)_-^p|\nabla\xi_j|^p\,dx\,dt,\\
J_3&=\esssup_{(x,t)\in\mathrm{spt}\,\xi_j,\,t\in\Gamma_j}\int_{{\mathbb{R}^N\setminus B_j}}{\frac{(u-k_j)_{-}(y,t)^{p-1}}{|x-y|^{N+ps}}}\,dy
\int_{\Gamma_j}\int_{B_j}(u-k_j)_{-}\xi_{j}^p\,dx\,dt,\\
\quad\text{ and }\\
J_4&=\int_{\Gamma_j}\int_{B_j}|g(x,t)|\,|u|^{p-1}\,(u-k_j)_-\,\xi_j^{p}\,dx dt
\end{split}
\end{equation*}
for some positive constant $C=C(N,p,s,C_1,C_2)$. Now, proceeding as in the proof of Lemma \ref{DGL1} above, we obtain by setting $Y_j=\frac{|A_j|}{|\mathcal{Q}_j|}$ that
$$
Y_{j+1}\leq c_0 b^{j}Y_j^{1+\delta},
$$
where
$$
\quad\delta=\frac{1}{N+2},\quad b=2^{\frac{(N+p)^2}{(N+2)p}+1}>1
$$
and
$$
c_0=\frac{C\Big(L^p+L\,S^{p-1}\Big)}{(1-a)M}\Big(\frac{\theta}{(1-a)\,\min\{S^{\alpha-1},\, M^{\alpha-1}\}}\Big)^\frac{1}{N+2}.
$$
Letting
\[
d_0=\frac{C\Big(L^p+L\,S^{p-1}\Big)}{(1-a)M}\Big(\frac{1}{(1-a)\,\min\{S^{\alpha-1},\, M^{\alpha-1}\}}\Big)^\frac{1}{N+2},
\]
\[
b=2^{(N+p)^2+1},\quad\delta=\frac{1}{N+2}
\quad\text{and}\quad 
c_0=d_0\,\theta^\delta
\]
in Lemma \ref{iteration}, we have $\lim_{j\to\infty}Y_j\to 0$, if
$Y_0\leq\nu=c_{0}^{-\frac{1}{\delta}}b^{-\frac{1}{\delta^2}}$. 
Let $\beta\in(0,1)$, then choosing $\theta=\beta\,d_0^{-\frac{1}{\delta}}\,b^{-\frac{1}{\delta^2}}$, which depends on $a,M,\mu^-,\lambda^-,N,p,s,\alpha,g,C_1,C_2$, we get $\nu=\beta^{-1}>1$. 
Hence the fact  that $Y_0\leq 1$ and thus Lemma \ref{iteration} imply that
$\lim_{j\to\infty}Y_j\to 0$.
Therefore, we have 
\[
u\geq\mu^-+aM
\text{ in }
\mathcal{Q}_{\frac{3r}{4},\theta}^{+}(x_0,t_0).
\]
Hence the result follows.
\end{proof}

\section{Acknowledgment} The author thanks IISER Berhampur for the seed grant: IISERBPR/RD/OO/2024/15, Date: February 08, 2024.

\noindent {\textsf{Prashanta Garain\\Department of Mathematical Sciences\\
Indian Institute of Science Education and Research Berhampur\\
Berhampur, Odisha 760010, India}\\ 
\textsf{e-mail}: pgarain92@gmail.com\\

\end{document}